\documentclass{amsart}
\usepackage{amssymb,amsmath,amsthm,verbatim,enumerate,IEEEtrantools, url}
\newtheorem{theorem}{Theorem}[section]

\newtheorem{proposition}[theorem]{Proposition}

\newtheorem{lemma}[theorem]{Lemma}

\newcommand{\abs}[1]{\left| #1 \right|}
\numberwithin{equation}{section}
\title{Large Sets Avoiding Patterns}
\author{Robert Fraser \and Malabika Pramanik}
\begin{document}
\maketitle
\begin{abstract}
We construct subsets of Euclidean space of large Hausdorff dimension and full Minkowski dimension that do not contain nontrivial patterns described by the zero sets of functions. The results are of two types. Given a countable collection of $v$-variate vector-valued functions $f_q : (\mathbb{R}^{n})^v \to \mathbb{R}^m$ satisfying a mild regularity condition, we obtain a subset of $\mathbb{R}^n$ of Hausdorff dimension $\frac{m}{v-1}$ that avoids the zeros of $f_q$ for every $q$. We also find a set that simultaneously avoids the zero sets of a family of uncountably many functions sharing the same linearization. In contrast with previous work, our construction allows for non-polynomial functions as well as uncountably many patterns. In addition, it highlights the dimensional dependence of the avoiding set on $v$, the number of input variables.  
\end{abstract}
{\allowdisplaybreaks 
\section{Introduction}
\noindent Identification of geometric and algebraic patterns in large sets has been a focal point of interest in modern analysis, geometric measure theory and additive combinatorics. A fundamental and representative result in the discrete setting that has been foundational in the development of a rich thory is Szemer\'edi's theorem \cite{S75}, which states that every subset of the integers with positive asymptotic density contains an arbitrarily long arithmetic progression. There is now an abundance of similar results in the continuum setting, all of which guarantee existence of configurations under appropriate assumptions on size, often stated in terms of Lebesgue measure, Hausdorff dimension or Banach density. While this body of work has contributed significantly to our understanding of such phenomena, a complete picture concerning existence or avoidance of patterns in sets is yet to emerge. In this paper, we will be concerned with the ``avoidance'' aspect of the problem. Namely, given a function $f: \mathbb{R}^{n v} \to \mathbb{R}^m$ satisfying certain conditions, how large  a set $E \subset \mathbb{R}^n$ can one construct that carries no nontrivial solution of the equation $f(x_1, \ldots, x_v) = 0$? In other words, we aim to find as large a set $E$ as possible such that $f(x_1, \ldots, x_v)$ is nonzero for any choice of distinct points $x_1, \ldots, x_v \in E$. 
\vskip0.1in
\noindent In the discrete regime, results of this type can be traced back to Salem and Spencer \cite{SS42} and Behrend \cite{B46}, who identify large subsets of the integers avoiding progressions. The Euclidean formulation of this problem appears to be of relatively recent vintage. In \cite{K98}, Keleti constructs a subset $E$ of the real numbers of full Hausdorff dimension avoiding all nontrivial ``one-dimensional rectangles''. More precisely, this means that there exist no solutions of the equation $x_2 - x_1 - x_4 + x_3 = 0$ with $x_1 < x_2 \leq x_3 < x_4$, $x_i \in E$, $1 \leq i \leq 4$. In particular, such a set contains no nontrivial arithmetic progression, as can be seen by setting $x_2 = x_3$. A counterpoint to \cite{K98} is a result of \L aba and the second author \cite{LP09}, who have established existence of three-term  progressions in special ``random-like'' subsets of $\mathbb{R}$ that support measures satisfying an appropriate ball condition and a Fourier decay estimate. Higher dimensional variants of this theme may be found in \cite{CLP, HLP}. On the other hand, large Hausdorff dimensionality, while failing to ensure specific patterns, is sometimes sufficient to ensure existence or even abundance of certain configuration classes; see for instance the work of Iosevich et al \cite{{GI12}, {GILP15}, {BIT16}, {GIP16}}. Harangi, Keleti, Kiss, Maga, M\'ath\'e, Mattila, and Strenner in \cite{HKKMMMS13} show that sets of sufficiently large Hausdorff dimension contain points that generate specific angles. 
\vskip0.1in
\noindent Nonexistence of patterns such as the one proved by Keleti \cite{K98} is the primary focus of this article. A main contribution of \cite{K98} is best described as a Cantor-type construction with memory, where selection of basic intervals at each stage is contingent on certain selections made at a much earlier step of the construction, so as to defy certain algebraic relations from taking place. This idea has been instrumental in a large body of subsequent work involving nonexistence of configurations. For example, in \cite{K08}, Keleti uses this to show that for any countable set $A$, it is possible to construct a full-dimensional subset $E$ of $\mathbb{R}$ such that 
\[x_2 - x_1 + a(x_3 - x_2) = 0 \]
has no solutions for any $a \in A$ where $x_1, x_2$ and $x_3$ are distinct points in $E$. Maga \cite{M10} exploits this idea to demonstrate a full-dimensional subset $E \subset \mathbb{R}^n$ not containing the vertices of any parallelogram. Other results in this direction of considerable generality, extending their predecessors in \cite{{K98}, {K08}, {M10}}, are due to M\'ath\'e \cite{M12}. Given any countable collection of polynomials $p_j: \mathbb R^{nm_j} \rightarrow \mathbb R$ of degree at most $d$ with rational coefficients, the main result of \cite{M12} ensures the existence of a subset $E \subseteq \mathbb{R}^n$ of Hausdorff dimension $\frac{n}{d}$ such that $p_j(x_1, \ldots, x_{m_j})$ is nonzero for any choice of distinct points $x_1, \ldots, x_{m_j} \in E$. The same conclusion continues to hold if the polynomials $p_j$ are replaced by $p_j(\Phi_{j,1}(x_1), \cdots, \Phi_{j,m_j}(x_{m_j}))$, where $\Phi_{j,k}$ are $C^1$-diffeomorphisms of $\mathbb R^n$. Interestingly, the Hausdorff dimension bound in \cite{M12}, while depending on the ambient dimension $n$ and the maximum degree $d$ of the polynomials, is independent of the number of input vectors $m_j$ in $p_j$, which may continue to grow without bound.
\vskip0.1in
\noindent This paper uses similar ideas to present two results in a somewhat different direction. The first complements M\'ath\'e's result mentioned above. It applies to a countable family of functions $f:\mathbb R^{nv} \rightarrow \mathbb R$ with a fixed $v$ that are not necessarily polynomials with rational coefficients. Further, in contrast with \cite{M12}, the Hausdorff dimension of the obtained set depends on the number of vector variables $v$. The second result is of a perturbative flavour, and gives a set of positive Hausdorff dimension that simultaneously avoids zeros of \emph{all} functions with a common linearization and bounded higher-order terms. To the best of our knowledge, such uniform avoidance results are new. Some points of tenuous similarity may be found in \cite{HKKMMMS13}, where the authors construct sets that avoid angles within a specific range, but the ideas, methods and goals are very different.   

\subsection{Main results} 
Our first result is most general in dimension one, where we need very mild restrictions on the functions whose zeros we want to avoid. The  higher-dimensional, vector-valued version of this result applies with some additional restrictions. We state these two separately.  
\begin{theorem} \label{non-simul-thm}
For any $\eta > 0$ and integer $v \geq 3$, let $f_{q}: \mathbb{R}^v \to \mathbb{R}$ be a countable family of functions in $v$ variables with the following properties: 
\begin{enumerate}[(a)]
\item There exists $r_q < \infty$ such that $f_q \in C^{r_q}([0, \eta]^v)$,  
\vskip0.1in 
\item For each $q$, some partial derivative of $f_q$ of order $r_{q} \geq 1$ does not vanish at any point of $[0, \eta]^v$. 
\end{enumerate} 
\vskip0.1in 
\noindent Then there exists a set $E \subseteq [0, \eta]$ of Hausdorff dimension at least $\frac{1}{v-1}$ and Minkowski dimension 1 such that $f_{q}(x_1, \ldots, x_v)$ is not equal to zero for any $v$-tuple of distinct points $x_1, \ldots, x_v \in E$ and any function $f_{q}$.
\end{theorem}
\vskip0.1in 
\begin{theorem} \label{non-simul-thm vector-valued}
Fix $\eta > 0$ and positive integers $m,n,v$ such that $v \geq 3$, and $m \leq n(v-1)$. Let $f_{q} : (\mathbb{R}^{n})^v \to \mathbb{R}^m$ be a countable family of $C^{2}$ functions with the following property: the derivative $Df_q$ has full rank on the zero set of $f_q$ for every $q$ on $[0, \eta]^{nv}$. 
\vskip0.1in
\noindent Then there exists a set $E \subseteq [0, \eta]^{n}$ of Hausdorff dimension at least $\frac{m}{v-1}$ and Minkowski dimension $n$ such that $f_{q}(x_1, \ldots, x_v)$ is not equal to zero for any $v$-tuple of distinct points $x_1, \ldots, x_v \in E^n$ and any function $f_{q}$.
\end{theorem}
\vskip0.1in 
{\em{Remarks: }} \begin{enumerate}[(a)]
\item If one seeks to avoid zeros of say a single function $f$, Theorem \ref{non-simul-thm} is nontrivial only when the components of $\nabla f(x)$ sum to zero at every point $x$ in the zero set of $f$. If this is not the case, then there is necessarily some interval $I$ such that $f(x_1, \ldots, x_v)$ is nonzero for points $x_i$ in the interval $I$.
\vskip0.1in
\item The points $x_1, \cdots, x_v \in E$ that ensure $f(x_1, \ldots, x_v) \ne 0$ in Theorems \ref{non-simul-thm} and \ref{non-simul-thm vector-valued} are taken to be distinct. This assumption, while needed for the proof, is nonrestrictive for purpose of applications. In fact, one can typically augment the family $\{f_q\}$ by $\{g_q\}$, where the function $g_q$ equals $f_q$ with certain input variables coincident. For instance,  Keleti's function $f(x_1, x_2, x_3, x_4) = (x_2 - x_1) - (x_4 - x_3) = -x_1 + x_2 + x_3 - x_4$ identifies ``one-dimensional rectangles'' in general,  and three-term arithmetic progressions only if $x_2 = x_3$. In order to obtain a set using our set-up that avoids both, we would need to apply our Theorem \ref{non-simul-thm} to the collection $\{f, g\}$, where $g(x_1, x_2, x_3, x_4) =  f(x_1, x_2, x_2, x_4) = -x_1 + 2x_2 - x_4$. 
\vskip0.1in
\item While Theorems \ref{non-simul-thm} and \ref{non-simul-thm vector-valued} are sharp in certain instances, for example when $m = n(v-1)$, this need not be the case in general, as Keleti's example shows (our result would only ensure a set of Hausdorff dimension $1/3$). Even though our results do not recover those of \cite{{K08},{M10},{M12}} in all instances where these results are applicable, the Hausdorff dimension provided in Theorems \ref{non-simul-thm} and \ref{non-simul-thm vector-valued} offers new bounds in settings where previously none were available, for instance where the functions are non-polynomials with mild regularity. It also improves the bound given in \cite{M12} for polynomials with rational coefficients in the regime where the degree $d$ is much larger than the number of variables $v$. On the other hand for polynomials of low degree in low dimensions, the Hausdorff dimension obtained in \cite{M12} improves ours, obtaining the best bound when $d=1$. It is interesting to note that despite the differences in the results, the proof has many points of similarity with earlier work.
\vskip0.1in 
\item Another point worth noting is that for the quadratic polynomial $f(x_1, x_2, x_3) = (x_3-x_1) - (x_2-x_1)^2$, the set in $\mathbb R$ avoiding zeros of $f$ is guaranteed to be of Hausdorff dimension $\frac{1}{2}$, both according to \cite{M12} and Theorem \ref{non-simul-thm}. It is not known however whether this bound is optimal.  
\end{enumerate} 
\vskip0.2in
Our second result is about a set on which no function $f$ with a given linearization and controlled higher order term is zero.
\begin{theorem}\label{simul-thm} 
Given any constant $K > 0$ and a vector $\alpha \in \mathbb R^v$ such that 
and such that
\begin{equation} \label{alpha-cancellation}
\sum_{j=1}^v \alpha_j = 0,
\end{equation}  
there exists a positive constant $c(\alpha)$ and a set $E = E(K, \alpha) \subseteq [0,1]$ of Hausdorff dimension $c(\alpha) > 0$ with the following property. 
\vskip0.1in
\noindent The set $E$ does not contain any nontrivial solution of the equation \[f(x_1, \cdots, x_v) = 0, \qquad (x_1, \cdots, x_v) \text{ not all identical}, \] for any $C^2$ function $f$ of the form \begin{align} f(x_1, \cdots, x_v) &= \sum_{j=1}^{v} \alpha_j x_j + G(x_1, \cdots, x_v)  \label{f-form} \\ \text{  where } |G(x)| &\leq K \sum_{j=2}^{v}(x_j - x_1)^2.  \label{G}\end{align} 
\end{theorem} 
\vskip0.1in
{\em{Remarks:}} \begin{enumerate}[(a)]
\item The condition \eqref{alpha-hypothesis} implies that $\alpha$ does not lie in any coordinate hyperplane.
\vskip0.1in
\item The proof of Theorem \ref{simul-thm} can be used to obtain a corresponding result with \emph{finitely} many linearizations. There is a loss in the Hausdorff dimension as more linear functions are added to the family, so the proof fails for families of functions with countably many linearizations.
\vskip0.1in
\item It is interesting to note that the dimensional constant $c(\alpha)$ does not depend on $K$. Of course the set $E$ does, and is uniform for all functions $f$ obeying \eqref{f-form} and \eqref{G} with a fixed value of $K$. 
\end{enumerate} 
\subsection{Layout} Section \ref{ex-section} is devoted to geometric applications of Theorems \ref{non-simul-thm}, \ref{non-simul-thm vector-valued} and \ref{simul-thm}. Optimality of these results (or lack theoreof) in various settings are discussed, and comparison with earlier work presented. Section \ref{Building-block-section} is a collection of geometric algorithms needed for the proofs of Theorems \ref{non-simul-thm} and \ref{non-simul-thm vector-valued}. The proofs themselves are executed in Sections \ref{Proofs of non-simul thms} and \ref{Proof of simul-thm}.  
  
\section{Examples}\label{ex-section}
\subsection{Subsets of curves avoiding isosceles triangles}
This subsection is given over to the following question: suppose we are given a small segment of a simple $C^2$ curve $\Gamma \subset \mathbb{R}^n$ with nonvanishing curvature $K$, parametrized by a $C^2$-function $\gamma : [0,\eta] \to \mathbb{R}^n$ with nonvanishing derivative. How large can the Hausdorff dimension of a subset $E \subseteq [0, \eta]$ be if there do not exist three points $x_1, x_2, x_3 \in E$ such that $\{\gamma(x_1), \gamma(x_2), \gamma(x_3)\} \subseteq \Gamma$ are the vertices of an isosceles triangle? 
\vskip0.1in
\noindent The existence of an isosceles triangle with vertices on $\Gamma$ will be determined using one of the following functions: 
\begin{equation} \label{def-f1} f_1(t_1, t_2, t_3) = |\gamma(t_1) - \gamma(t_2)|^2 - |\gamma(t_2) - \gamma(t_3)|^2, \end{equation} 
or 
\begin{equation} \label{def-f2} f_2(t_1, t_2, t_3) = d(\gamma(t_1), \gamma(t_2)) - d(\gamma(t_2), \gamma(t_3)). \end{equation} 
Here $d$ is the ``signed distance" along the curve $\Gamma$ defined by 
\begin{equation} \label{def-d} d(\gamma(t_1), \gamma(t_2)) = \begin{cases} |\gamma(t_1) - \gamma(t_2)| &\text{ if } t_1 > t_2 \\  - |\gamma(t_1) - \gamma(t_2)| &\text{ if } t_1 < t_2. \end{cases} \end{equation} 
For reasons to be explained shortly, we will want to avoid the zero set of $f_1$ or $f_2$. In order to apply Theorem \ref{non-simul-thm}, we need to verify that these functions are differentiable. This is evident for $f_1$. In Lemma \ref{d-differentiable} of the appendix, we have shown that the signed distance $d$ is differentiable, which provides the same conclusion for $f_2$.    
\vskip0.1in 
\noindent Let $f$ be either the function $f_1$ or $f_2$ given in \eqref{def-f1} or \eqref{def-f2}. In either case, we have that if $f(t_1, t_2, t_3) = 0$, then $\gamma(t_1), \gamma(t_2), \gamma(t_3)$ form the vertices of an isosceles triangle or points on an arithmetic progression. Conversely, let $x, y, z$ be distinct points of $\Gamma$ that form an isosceles triangle, with $|x-y| = |y-z|$. Then there exist $t_1 < t_2 < t_3$ such that some permutation of $\gamma(t_1), \gamma(t_2), \gamma(t_3)$ will be the points $x, y, z$. It is not difficult to see that if $\eta$ is sufficiently small depending on $|\gamma'(0)|$ and the curvature $K$, then $y$ can neither be $\gamma(t_1)$ or $\gamma(t_3)$. We include a proof of this in Lemma \ref{order-lemma} the appendix. Therefore $y = \gamma(t_2)$, in which case $f(t_1, t_2, t_3) = 0$.  

\subsubsection{A set avoiding isosceles triangles along a single curve}
We will first discuss the problem of avoiding isosceles triangles along a single curve $\Gamma$. For this variant of the problem, $\gamma$ may be any parametrization of $\Gamma$ satisfying the conditions laid out above. 
\vskip0.1in 
\noindent Let us first consider the case where $\Gamma$ is parameterized by a polynomial function $\gamma$ of degree $d$ with rational coefficients, i.e., $\gamma(t) = (p_1(t), p_2(t), \ldots, p_n(t))$. Let us observe that the result in \cite{M12}  does not apply to the non-polynomial function $f_2(t_1, t_2, t_3)$, but does apply to \begin{multline*} f_1(t_1, t_2, t_3) = \bigl[(p_1(t_1) - p_1(t_2))^2 + \cdots + (p_n(t_1) - p_n(t_2))^2\bigr] \\ - \bigl[(p_1(t_2) - p_1(t_3))^2 + \cdots + (p_n(t_2) - p_n(t_3))^2 \bigr],\end{multline*} which is a polynomial of degree at most $2d$. Applying \cite{M12} then gives a subset of $\Gamma$ of Hausdorff dimension $\frac{1}{2d}$ that does not contain the vertices of any isosceles triangle.
\vskip0.1in 
\noindent If $\Gamma$ is a general (not necessarily polynomial) $C^2$ curve with parameterization $\gamma(t)$, and $f(t_1, t_2, t_3)$ is either $f_1$ or $f_2$ described above, then Theorem \ref{non-simul-thm} demonstrates the existence of a subset $E$ of $[0,1]$ of Hausdorff dimension $\frac{1}{2}$ such that $f(t_1, t_2, t_3) \neq 0$ for any choice of $t_1, t_2, t_3 \in E$. Under $\gamma$, this lifts to a subset of $\Gamma$ of Hausdorff dimension $\frac{1}{2}$ that does not contain the vertices of an isosceles triangle. Even for the case of functions with a rational polynomial parametrization, this set has a larger Hausdorff dimension than the one provided by \cite{M12}.
\vskip0.1in
\noindent Incidentally, it is instructive to compare the above with the case where the curve $\gamma$ is a line, even though the curvature for the latter is zero. Here we will view three term arithmetic progressions as degenerate isosceles triangles. Set $\gamma(t) = at + b$ for some $a, b \in \mathbb R^n$, $a \ne 0$. Then the function $f(t_1, t_2, t_3) = t_1 + t_3 - 2t_2$ is equal to zero precisely when $\gamma(t_1), \gamma(t_2)$ and $\gamma(t_3)$ lie in arithmetic progression. Keleti's result \cite{K98} as well as \cite{M12} applied to this $f$ shows that there is a subset of $\Gamma$ of Hausdorff dimension $1$ that does not contain any arithmetic progressions. Theorem \ref{non-simul-thm} on the other hand provides a set with nonoptimal Hausdorff dimension $1/2$.

\subsubsection{A set avoiding isosceles triangles along all curves with bounded curvature} \label{example-all} 
We will also ask a question related to the one above, this time considering only $C^2$ curves given by arclength parametrization. How large a set $E \subset [0,1]$ can we construct such that $\gamma(E)$ does not contain any isosceles triangle for any $\gamma : [0,1] \to \mathbb{R}^n$ with $|\gamma^{\prime}(t)| \equiv 1$ and with curvature at most $K$? 
\vskip0.1in
\noindent For any such curve $\gamma$, the function $f_2$ defined in \eqref{def-f2} will be differentiable everywhere, with $\frac{\partial f_2}{\partial t_1} = \frac{\partial f_2}{\partial t_3} \equiv 1$ and $\frac{\partial f_2}{\partial t_2} \equiv -2$, as we have verified in Lemma \ref{d-differentiable} part(\ref{special-gamma}). Thus the function $f_2$ will satisfy the conditions of Theorem 1.3. One therefore obtains a subset $E \subset [0,1]$ of positive Hausdorff dimension such that $f_2(t_1, t_2, t_3) \ne 0$ whenever $t_1, t_2, t_3 \in E$ are distinct, no matter which $\gamma$ we choose in this class. Thus the points parametrized by $E$ manage to avoid isosceles triangles on all curves $\Gamma$ with a fixed bounded curvature.   
\vskip0.1in
\noindent How large a Hausdorff dimension can we get? A careful scrutiny of Lemma \ref{simul-lemma}, Proposition \ref{mainprop-simul} and Theorem \ref{simul-thm} shows that one can ensure sets of Hausdorff dimension at least $\frac{\log 2}{\log 3}$. For more details, we refer the reader to the proofs of these results in Section \ref{simul-thm-proof} and the remarks following them.  

\subsubsection{Discussion on optimality} \label{discussion-optimality}
Clearly Theorem \ref{non-simul-thm vector-valued} is optimal when $m = n(v-1)$. On the other hand, 
we can use Theorem \ref{simul-thm} together with the example above to give a polynomial with rational coefficients for which neither \cite{M12} nor Theorem \ref{non-simul-thm} give the optimal bounds. Consider a polynomial of the form \[p(t_1, t_2, t_3) = t_1 - 2 t_2 + t_3 + q(t_1, t_2, t_3)\] where $q(t_1, t_2, t_3)$ is a nontrivial homogeneous quadratic polynomial in $(t_2 - t_1)$ and $(t_3 - t_1)$ with rational coefficients. We are of course interested in finding a set $E$ (as large as possible) such that $p(t_1, t_2, t_3) \ne 0$ for any choice of distinct points $t_1, t_2, t_3 \in E$. Both \cite{M12} and Theorem \ref{non-simul-thm}  provide such a set $E$, with  dimension at least $1/2$ in both cases. Theorem \ref{simul-thm} provides such a set $E$ as well. Note that $p$ has the same linearization as the functions $f$ described in the previous section above. Hence, as described at the end of Section \ref{example-all}, the set $E$ obtained via Theorem \ref{simul-thm} is a set of dimension at least $\frac{\log 2}{\log 3} > 1/2$, proving the claimed suboptimality statement. 
\vskip0.1in
\noindent In fact, we can use this framework to construct other examples. Notice that it is possible to ask for sets $E$ that avoid triangles that are not necessarily isosceles, for instance triangles where the sidelength ratio is a prescribed constant $\kappa$. The results in \cite{{K98}, {M12}} and Theorem 1.1 all apply to give a set with the same Hausdorff dimension 1/2 as above not containing $t_1, t_2, t_3$ such that $|\gamma(t_2) - \gamma(t_1)| = \kappa |\gamma(t_3) - \gamma(t_1)|$. However, the Hausdorff dimension bound in Theorem \ref{simul-thm} becomes worse as $\kappa$ moves farther away from $1$. Still, for $\kappa$ close to $1$, Theorem \ref{simul-thm} outperforms Theorem \ref{non-simul-thm}, giving rise to a family of polynomials whose zeros can be avoided by a set of unusually large Hausdorff dimension. 

\subsection{A subset of a curve not containing certain kinds of trapezoids} \label{trapezoid-section}
The following is a geometric example of Theorem 1.2. Call a trapezoid $ABCD$ with $AD$ parallel to $BC$ ``special" if the sidelengths obey the restriction $|BC|^2 = |AB||CD|$.  Given a curve $\Gamma \subset \mathbb{R}^2$ parametrized by a function $\gamma:[0, \eta] \rightarrow \mathbb R$, we aim to find a subset $E$ of $[0, \eta]$ with the following property: for any choice of $t_1 < t_2 < t_3 < t_4$ in $E$, the trapezoid $ABCD$ with 
\[A = \gamma(t_1), \quad B = \gamma(t_2), \quad C = \gamma(t_3), \quad D = \gamma(t_4)\]
is not special. For simplicity and ease of exposition, we may assume that the components of $\gamma'$ are strictly positive on $[0, \eta]$ and that the curvature is also of constant sign, say $\Gamma$ is strictly convex.  
\vskip0.1in 

\noindent Notice that the special trapezoid assumption places two essentially independent conditions on $\gamma(t_1), \gamma(t_2), \gamma(t_3)$, and $\gamma(t_4)$.  One is that two sides need to be parallel and the other is the condition on the sidelengths. Accordingly we define two functions $f_1$ and $f_2$ as follows:  
\begin{align} f_1(t_1, t_2, t_3, t_4) &= \text{arg}(\gamma(t_4) - \gamma(t_1)) - \text{arg}(\gamma(t_3) - \gamma(t_2)) \label{def-f1-trapezoid} \\  
f_2(t_1, t_2, t_3, t_4) & = d(\gamma(t_4), \gamma(t_3)) d(\gamma(t_2), \gamma(t_1)) - d(\gamma(t_3), \gamma(t_2))^2. \label{def-f2-trapezoid}
\end{align} 
Here ``arg'' denotes the argument function defined as follows: for any $p \in \mathbb R^2 \setminus \{ p = (x,0): x \leq 0 \}$, the quantity arg$(p)$ is the unique angle between $(-\pi, \pi)$ that the line joining $(x,y)$ to the origin makes with the positive $x$ axis. Clearly, $f_1$ is zero if and only if $AD$ is parallel to $BC$, while $f_2$ is zero if and only if $|BC|^2 = |AB||CD|$. We therefore seek to avoid the zeros of $f = (f_1, f_2)$. We verify in Lemma \ref{Jacobian-check} of the appendix that the derivative $Df$ is of full rank on the zero set of $f$. Applying Theorem \ref{non-simul-thm vector-valued} with $n = 1, m=2, v = 4$, we obtain a set $E$ of Hausdorff dimension $2/3$ such the points on $\Gamma$ indexed by $E$ avoids special trapezoids as explained above.
%
%
%
Thus, there is a subset of $\Gamma$ of Hausdorff dimension $\frac{2}{3}$ that does not contain any special trapezoids.

\section{Avoidance of zeros on a single scale} \label{Building-block-section}
\noindent The proofs of Theorems \ref{non-simul-thm} and \ref{non-simul-thm vector-valued} are based on an iterative construction, whose primary building block relies on an algorithm: given a set $T \subseteq \mathbb R^{nv}$ contained in the domain of a suitably nonsingular function $f: \mathbb R^{nv} \rightarrow \mathbb R^m$, one identifies a subset $S \subseteq T$ that stays away from the zero set of $f$. This zero-avoiding subset $S$, which is a union of cubes in $\mathbb R^{nv}$ (and as such of positive Lebesgue measure and full Hausdorff dimension), does not immediately yield the set we seek because it is typically not the $v$-fold Cartesian product of a set in $\mathbb R^n$ with itself, and hence does not meet the specifications of the theorems.  However, the algorithm can be used iteratively on many different scales and for many functions in the construction of the set $E$ whose existence has been asserted in the theorems. Our objective in this section is to describe this algorithm. The versions that we need for Theorems \ref{non-simul-thm} and \ref{non-simul-thm vector-valued} are very similar in principle, although the exact statements differ somewhat. These appear in Propositions \ref{non-simul-mainprop} and \ref{non-simul-mainprop vector-valued} below respectively. 

\subsection{Building block in dimension one}

\noindent Let $f$ be a real-valued $C^1$ function of $v$ variables and nonvanishing gradient defined in a neighbourhood of the origin containing $[0,1]^v$. Suppose that we are given an index $i_0 \in \{1, 2, \cdots, v\}$, an integer $M \geq 1$, a small constant $c_0 > 0$ and compact subsets $T_1, \cdots, T_v \subseteq [0,1]$ with the following properties: 
\begin{align}
& \begin{aligned} &\text{$\bullet$ Each $T_i$ is a union of closed intervals of length $M^{-1}$ with disjoint} \\ &\text{interiors. Let us denote by $\mathcal J_M(T_i)$ this collection of intervals.} \end{aligned} \label{T-union} \\  
& \bullet \Bigl|\frac{\partial f}{\partial x_{i_0}}(x) \Bigr| \geq c_0 \text{ and } |\nabla f(x)| \leq c_0^{-1} \text{ for all } x \in T_1 \times \cdots \times T_v. \label{deriv-nonzero}
\end{align}       
\begin{proposition} \label{non-simul-mainprop}
Given $f, M, i_0, c_0$ and $\mathbb T = (T_1, \cdots, T_v)$ obeying \eqref{T-union} and \eqref{deriv-nonzero} above, there exist a constant $c_1 > 0$ and an integer $N_0$ depending on all these quantities such that for all integers $N \geq N_0$ that are integer multiples of $M$, the following conclusions hold. 
There exist compact subsets $S_i \subseteq T_i$ for all $1 \leq i \leq v$ such that  
\begin{enumerate}[(a)]
\item There are no solutions of $f(x) = 0$ with $x \in S_1 \times \cdots \times S_v$. \label{no solutions}
\vskip0.1in
\item For each $J \in \mathcal J_M(T_i)$, let us decompose $J$ into closed intervals of length $N^{-1}$ with disjoint interiors and call the resulting collection of intervals $\mathcal I_N(J,i)$. Then for each $i \ne i_0$ and each  $I \in \mathcal I_N(J, i)$, the set $S_i \cap I$ is an interval of length $c_1 N^{1-v}$.   \label{inoti0}
\vskip0.1in 
\item For every $J \in \mathcal J_{M}(T_{i_0})$, there exists $\mathcal I_{N}'(J, i_0) \subseteq \mathcal I_N(J, i_0)$ with 
\begin{equation} \label{I_N'I_N} \# \bigl( \mathcal I_{N}'(J, i_0) \bigr) \geq \bigl(1 - \frac{1}{M} \bigr) \# \bigl( \mathcal I_{N}(J, i_0) \bigr) \end{equation} 
such that for each $I \in \mathcal I_N'(J, i_0)$, 
\begin{equation}  |S_{i_0} \cap I| \geq \frac{c_1}{N}. \label{S_i0}\end{equation} 
Unlike part (\ref{inoti0}), $S_{i_0} \cap I$ need not be an interval; however, it can be written as a union of intervals of length $c_1 N^{1-v}$ with disjoint interiors.       \label{i0}
\end{enumerate}
\end{proposition}
\begin{proof}
Without loss of generality, we may set $i_0 = v$. For $i \ne v$, we define 
\[ S_i = \bigcup \left\{  [a_i, a_i + c_1 N^{1-v}]\, : \, [a_i, b_i] = I \in \mathcal I_N(J, i) \text{ for some } J \in \mathcal J_M(T_i) \right\}, \]
where the small positive constant $c_1$ and the integer $N$ will be specified shortly. In other words, $S_i$ consists of the leftmost $c_1N^{1-v}$-subintervals of all the $1/N$-intervals that constitute $T_i$. It is clear that the conclusion (\ref{inoti0}) holds for this choice of $S_i$. 
\vskip0.1in

\noindent We now proceed to  define the subcollection $\mathcal I_N'(J,v)$ and the set $S_v$ that obey the requirements in (\ref{i0}).  Consider the collection \[ \mathbb A_N := \prod_{i=1}^{v-1} \bigl\{a_i : [a_i, b_i] = I \in \mathcal I_N(J, i) \text{ for some } J \in \mathcal J_M(T_i) \bigr\}  \] 
consisting of $(v-1)$-tuples of the form $a' = (a_1, \cdots, a_{v-1})$, where each $a_i$ is a left endpoint of an interval in $\mathcal I_N(J, i)$, for some $J \in \mathcal J_{M}(T_i)$. For each $i$, the number of possible choices for $1/N$-intervals $I \subseteq [0,1]$ and hence for $a_i$ is at most $N$. Thus 
\begin{equation} \label{A_n-card} 
\#(\mathbb A_N) \leq N^{v-1}.  
\end{equation} 
We will prove in Lemma \ref{cardinality for fixed a' lemma} below that for every fixed $a' \in \mathbb A_N$, 
\begin{equation} \label{v-proj}
\# \left\{x_v : f(a', x_v) = 0 \right\} \leq M.  
\end{equation} 
Assuming this for the moment, define \[\mathbb B := \{x_v : \exists \; a' \in \mathbb A_N \text{ such that } f(a', x_v)= 0  \}.\] In light of \eqref{A_n-card} and \eqref{v-proj}, we find that 
\begin{equation} \#(\mathbb B) \leq MN^{v-1}. \label{cardB} 
\end{equation}  
\vskip0.1in

\noindent The subcollection $\mathcal I'_{N}(J, v) \subseteq \mathcal I_{N}(J, v)$ specified in part (\ref{i0}) is chosen as follows: we declare \[ I \in \mathcal I_N'(J,v) \quad \text{ if } \quad \#(\mathbb B \cap I) \leq M^3 N^{v-2}. \] In view of \eqref{cardB} and the pigeonhole principle, it follows that 
\begin{equation} 
\# \bigl(\mathcal I_N(J,v) \setminus \mathcal I_N'(J,v) \bigr) \leq \frac{MN^{v-1}}{M^3 N^{v-2}} = \frac{N}{M^2}.  
\end{equation}  
The fact that $\#(\mathcal I_N(J,v)) = \frac{N}{M}$ then implies \eqref{I_N'I_N}.  
\vskip0.1in
 
\noindent We now decompose each $I \in \mathcal I_N'(J,v)$ into consecutive subintervals of length $N^{v-1}/(C_0c_1)$ with disjoint interiors, and denote the successive intervals by $\tilde{I}_\ell(I)$:   \[ I = \bigcup \{ \tilde{I}_\ell(I) : 1 \leq \ell \leq N^{v-2}/(C_0c_1)\}. \] Here $C_0$ is a constant integer depending on $f$, $M$ and $T_1, \cdots T_v$, as has been specified in Lemma \ref{implicit function lemma} below. The integer $N$ is chosen large enough so that $N^{v-2}/(C_0c_1)$ is an integer. All intervals $\tilde{I}_\ell(I)$ that intersect $\mathbb B$, together with their adjacent neighbours, are then discarded. The remaining subset of $T_v$ is defined to be $S_v$. More specifically, 
\[ S_v = \bigcup \left\{\tilde{I}_\ell(I) : \begin{aligned} &\tilde{I}_k(I) \cap \mathbb B = \emptyset \text{ for }  |k-\ell| \leq 1, \\ &I \in \mathcal I_N'(J, v), \;  J \in \mathcal J_M(T_v) \end{aligned} \right\}.  \]
Clearly $S_v$ can be viewed a union of intervals of length $c_1/N^{v-1}$. The definition of $\mathcal I_N'(J,v)$ implies that the total length of the discarded subintervals in each $I \in\mathcal I_N'(J,v)$ is at most $3C_0c_1M^3N^{v-2}/N^{v-1} = 3M^3C_0c_1/N$. The claim \eqref{S_i0} now follows by choosing $c_1 > 0$ small enough so as to satisfy $3M^3 C_0 c_1 < (1-c_1)$.     
\vskip0.1in

\noindent Finally, Lemma \ref{implicit function lemma} below shows that given $x' = (x_1, \cdots, x_{v-1}) \in S_1 \times S_2 \times \cdots \times S_{v-1}$, any $x_v$ obeying $f(x',x_v) = 0$ should necessarily lie within a $C_0c_1/N^{v-1}$ neighbourhood of $\mathbb B$. Since the set $S_v \subseteq T_v$ was created so as to avoid these neighbourhoods, conclusion (\ref{no solutions}) follows.   
\end{proof} 

\begin{lemma} \label{cardinality for fixed a' lemma}  
For $f$ and $\mathbb A_N$ as in Proposition \ref{non-simul-mainprop}, the inequality \eqref{v-proj} holds for every fixed $a' \in \mathbb A_N$. 
\end{lemma} 
\begin{proof} 
Given $a' \in \mathbb A_N$, we claim that for every $J \in \mathcal J_M(T_v)$, there exists at most one $x_v \in J$ such that $f(a', x_v) = 0$. Since the number of possible $J \in \mathcal J_M(T_v)$ is at most $M$, the desired conclusion would follow once the claim is established. 
\vskip0.1in

\noindent To prove the claim, let us assume if possible that there exist $x_v, y_v \in J$,  $x_v \ne y_v$ such that $f(a',x_v) = f(a', y_v) = 0$. By Rolle's theorem, this ensures the existence of some point $z_v \in J$ where $\partial f/\partial x_v (a', x_v) = 0$. But this contradicts the hypothesis  (\ref{deriv-nonzero}) that the partial derivative $\partial f/\partial x_v$ is nonzero on $T_1 \times \cdots \times T_v$. 
\end{proof} 

\begin{lemma} \label{implicit function lemma} 
Let $f$, $M$ and $T_1, \cdots, T_v$ be as in Proposition \ref{non-simul-mainprop}. Then there exists a constant $C_0$ depending on these quantities, and in particular on $c_0$ such that for the choice of $S_1, S_2, \cdots S_{v-1}$ as specified in the proof of the proposition,  
\[ \text{dist}(x_v, \mathbb B) \leq \frac{C_0c_1}{N^{v-1}} \] for any $x_v$ obeying $f(x) = 0$, with $x' = (x_1, \cdots, x_{v-1}) \in S_1 \times \cdots \times S_{v-1}$. 
\end{lemma} 
\begin{proof} 
Let $\mathbb J = J_1 \times \cdots \times J_v = \mathbb J' \times J_v \in \prod_{i=1}^v \mathcal J_M(T_i)$ be a $v$-dimensional cube of sidelength $1/M$ such that the zero set of $f$ intersects $\mathbb J$. The nonvanishing derivative condition (\ref{deriv-nonzero}) then implies, in view of the implicit function theorem, that there exists a $(v-1)$-variate $C^1$ function $g_{\mathbb  J}$ defined on $\mathbb J'$ and a constant $C_0 > 0$ depending on $c_0, M, T_1, \cdots, T_v$ such that 
\begin{align}
&f(x) = 0, \; x \in \mathbb J \quad \text{ implies } \quad x_v = g_{\mathbb J}(x'), \; x' \in \mathbb J', \quad \text{ and } \label{def-g_J}\\ 
& |\nabla g_{\mathbb J}| \leq \frac{C_0}{\sqrt{v}} \text{ on } \mathbb J'. \label{gradient-bound}  
\end{align}  
\vskip0.1in

\noindent Given $x = (x', x_v) \in S_1 \times \cdots \times S_v$ such that $f(x) = 0$, let $\mathbb J_x$ denote the $v$-dimensional $1/M$-cube $\mathbb J$ in which $x$ lies, and let $\mathbb I'_x = I_1 \times \cdots \times I_{v-1} = \prod_{i=1}^{v-1} [a_i, b_i] \in \prod_{i=1}^{v-1} \mathcal I_N(J_i, i)$ be the $(v-1)$-dimensional subcube of $\mathbb J_x'$ of sidelength $1/N$ containing $x'$. Then \[ x_v = g_{\mathbb J}(x'), \quad a' = (a_1, \cdots, a_{v-1}) \in \mathbb A_N, \quad g_{\mathbb J}(a') \in \mathbb B, \quad \text{ and } \quad |x' - a'| \leq \frac{c_1\sqrt{v}}{N^{v-1}}. \] Further, \eqref{gradient-bound} implies \begin{align*} \text{dist}(x_v, \mathbb B) \leq  |g_{\mathbb J}(a') - g_{\mathbb J}(x')| &\leq || \nabla g_{\mathbb J}||_{\infty} |x'-a'|  \\ &\leq (C_0/\sqrt{v}) \times (c_1 \sqrt{v}/N^{v-1}) = \frac{C_0c_1}{N^{v-1}}, \end{align*}  which is the conclusion of the lemma.       
\end{proof} 

\subsection{Building block in higher dimensions} 
Given positive integers $m,n \geq 1$ and $v \geq 3$ with $m \leq n(v-1)$, let $f : \mathbb{R}^{nv} \to \mathbb{R}^m$ be a $C^2$ function, whose derivative is of full rank, and whose zero set has nontrivial intersection with $[0,1]^{nv}$. Suppose that $M \geq M_0$ is a large integer, $c_0 > 0$ is a small constant and $T_1, \ldots, T_v \subseteq [0,1]^n$ are sets with the following properties: 
\begin{align} 
& \begin{aligned} \bullet  & {\text{ Each $T_i$ is expressible as a union of axis-parallel cubes of sidelength}} \\ 
                                         & {\text{ $M^{-1}$, the collection of which will be called $\mathcal J_M(T_i)$.}} 
\end{aligned}  \label{T-def-vector} \\
&  \begin{aligned} \bullet &{\text{ On $([0,1]^{n})^v$, the smallest singular value of the derivative $Df$ is }} \\ &{\text{ bounded above and below by $c_0^{-1}$ and $c_0$ respectively. }} \end{aligned}  \label{D-singular-value}\\
& \begin{aligned}  \bullet &{\text{ On $([0,1]^n)^v$, the matrix norm of the Hessian $D^2f$ is bounded}} \\ &{\text{ above by $c_0^{-1}$. }} \end{aligned} \label{D-Hessian-bound}
\end{align} 

\begin{proposition}\label{non-simul-mainprop vector-valued}
Given $f, M$ and $c_0$ as above, there exists a constant $c_1 > 0$ and an integer $N_0$ depending on these quantities such that for all $N \geq N_0$ that are integer multiples of $N_0$ the following conclusions hold. There are compact subsets $S_i \subseteq T_i$ for all $1 \leq i \leq v$ such that 
\vskip0.1in
\begin{enumerate}[(a)]
\item \label{f-nonzero} There are no solutions to $f(x) = 0$ with $x \in S_1 \times \cdots \times S_v$. \label{no solutions vector-valued}
\vskip0.1in 
\item \label{S_i-description} For each $1 \leq i \leq v$ and $J \in \mathcal J_M(T_i)$, let us decompose $J$ into closed axis-parallel cubes of length $N^{-1}$ with disjoint interiors and call the resulting collection of cubes $\mathcal I_N(J,i)$. There exists $\mathcal I_N'(J, i) \subseteq \mathcal I_N(J,i)$ such that 
\[ S_i \subseteq \bigcup \left\{I : J \in \mathcal J_M(T_i), \; I \in \mathcal I_N'(J,i) \right\}.  \]
More precisely, for each $I \in \mathcal I_N'(J,i)$, the set $S_i \cap I$ is a single axis-parallel cube of sidelength $\ell = c_1 N^{n(1-v)/m}$, provided $i \ne v$. For $i = v$ and $I \in \mathcal I_N'(J, v)$, the set $S_v \cap I$ is not necessarily a single cube of sidelength $\ell$, but a union of such cubes, with the property that    
\begin{equation}\label{vsize}
|S_{v} \cap I| \geq \left(1 - \frac{1}{M} \right) \frac{1}{N^n}. 
\end{equation} 
\vskip0.1in 
\item The subcollections $\mathcal I_N'(J,i)$ of cubes are large subsets of the ambient collection $\mathcal I_N(J, i)$, in the sense that for all $1 \leq i \leq v$, $J \in \mathcal J_M(T_i)$, 
\begin{equation} \label{I_N'I_N vector-valued} 
\# \bigl( \mathcal I_{N}'(J, i) \bigr) \geq \left (1 - \frac{1}{M} \right) \# \bigl( \mathcal I_{N}(J, i) \bigr). 
\end{equation} 
%
\end{enumerate}
\end{proposition}
\vskip0.1in 
{\em{Remarks: }} \begin{enumerate}[(a)]
\item The proof will show that the constant $c_1$ in Proposition \ref{non-simul-mainprop vector-valued} may be chosen as a small constant multiple of $M^{-R}$, where $R = [(n+1)v + 1]/m$. For the purposes of application, $M$ is negligible compared to $N$, and hence the specific power of $M$ that appears in the expression for $\ell$ is not critical to the proof. The power of $N$, which is $-\frac{n}{m}(v-1)$, is of utmost importance and the principal reason that the Hausdorff dimension of the set $E \subseteq \mathbb R^n$ in Theorem \ref{non-simul-thm vector-valued} is equal to $\frac{m}{v-1}$. 
\vskip0.1in
\item The restriction $m \leq n(v-1)$ justifies on one hand the dimensional constraint on the set $E$ which lies in $\mathbb R^n$. On a technical note, it is also necessary for the assumption $\ell \ll N^{-1}$ that permeates the proof. If $m < n(v-1)$, the chosen value of $\ell = \epsilon_0 M^{-R} N^{-n(v-1)/m}$ will be less than $\frac{1}{N}$ if $N$ is sufficiently large. If $m = n(v-1)$, the chosen value of $\ell$ will be less than $\frac{1}{N}$ provided that $M$ is sufficiently large.  
\vskip0.1in 
\item The special treatment of the variable $x_v$ in the proposition is for convenience only. The result holds for $x_v$ replaced by $x_{i_0}$, for any $1 \leq i_0 \leq v$.
\end{enumerate} 

\begin{proof}
Let $Z_f = \{x = (x_1, \cdots, x_v) \in ([0,1]^n)^v : f(x) = 0 \}$ be the zero set of the function $f$, which we wish to avoid. The assumptions \eqref{D-singular-value} and \eqref{D-Hessian-bound} ensure that $Z_f$ is an $(nv-m)$-dimensional submanifold of $[0,1]^{nv}$. Further, the co-area formula gives that $Z_f$ is coverable by at most $C \epsilon^{m-nv}$ many cubes of sidength $\epsilon$, for all sufficiently small $\epsilon$. Here $C$ is a large constant depending only on $c_0$ and independent of $\epsilon$. The proof consists of projecting $Z_f$ successively onto the coordinates $x_1, x_2, \cdots$, and selecting the sets $S_i$ so as to avoid the projected zero sets. The main ingredient of this argument has been described in Lemma \ref{proj-lemma}. We ask the reader to view the statement of the lemma first.  Assuming the lemma, the remainder of the proof proceeds as follows.   
\vskip0.1in 
\noindent Fix a parameter $\ell \ll 1/N$ soon to be specified. Recalling that $\mathcal I_{\alpha^{-1}}(J, i)$ denotes the collection of axis-parallel subcubes of sidelength $\alpha$ that constitute a partition of $J \in \mathcal J_M(T_i)$, let us define the collection of ``bad boxes'' $\mathbb{B}_1$ as follows: 
\begin{equation}
\mathbb{B}_1 = \{Q \in \prod_{i=1}^v \mathcal{I}_{\ell^{-1}}(J_i,i) : Q \cap Z_f \neq \emptyset \text{ for some $J_i \in J_M(T_i)$} \}.
\end{equation}
In other words, a box of sidelength $\ell$ in $T_1 \times \cdots \times T_v$ is considered bad if it intersects a point in the zero set of the function $f$. The discussion in the preceding paragraph shows that 
\begin{equation}  \label{B0-size} \#(\mathbb B_1) \leq C \ell^{m - nv} \end{equation}   where $C$ is a constant that depends only on the function $f$ and the value $c_0$. 
\vskip0.1in
\noindent The construction of $S_1, \cdots, S_v$ now proceeds as follows. At the first step, we project the boxes in $\mathbb B_1$ onto their $(x_2, \cdots, x_v)$ coordinates (each $n$-dimensional), and use Lemma \ref{proj-lemma} below with $r=v$, $T = T_1$, $T' = T_2 \times \cdots \times T_v$ and $\mathbb B = \mathbb B_1$ to arrive at a set $S_1 \subseteq T_1$ and a family of $n(v-1)$-dimensional boxes $\mathbb B' = \mathbb B_2$ obeying the conclusions of that lemma. Clearly the set $S_1$ obeys the requirements of part (\ref{S_i-description}) of the proposition.   
Lemma \ref{proj-lemma} also ensures that 
\[ \#(\mathbb B_2) \leq M^{n+1} N^n \ell^n \#(\mathbb B_1) \leq M^{n+1} N^n \ell^{m-n(v-1)}, \]  
and that $f(x) \ne 0$ for any $x = (x_1, x')$ such that $x_1 \in S_1$ and any $x' \in T_2 \cdots \times T_v$ that is not contained in the cubes constituting $\mathbb B_2$.       
\vskip0.1in
\noindent We now inductively follow a procedure similar to the above. At the end of step $j$, we will have selected sets $S_1 \subseteq T_1, \cdots, S_j \subseteq T_j$ and will be left with a family $\mathbb{B}_{j+1}$ of $n(v-j)$-dimensional cubes of sidelength $\ell$, such that 
\begin{align} 
&\#(\mathbb B_{j+1}) \leq  M^{(n+1)j} N^{jn} \ell^{m - n(v-j)} \text{ and } \\
\label{induction-j} & \left\{ \begin{aligned} &f(x'', x^{\prime}) \ne 0 \text{ for } x'' = (x_1, \cdots, x_j) \in \prod_{i=1}^{j}S_i, \; x^{\prime} \in \prod_{i=j+1}^{v} T_i, \\  & x' \text{ not contained in any of the cubes in $\mathbb B_{j+1}$}. \end{aligned} \right\} \end{align}   We can then apply Lemma \ref{proj-lemma} with \[ T = T_{j+1}, \quad T^{\prime} = T_{j+2} \times \cdots \times T_v, \quad \mathbb{B} = \mathbb{B}_{j+1} \] to arrive at a set $S_{j+1} \subseteq T_{j+1}$ meeting the requirement of part (\ref{S_i-description}) of the proposition. 
The lemma also gives a family $\mathbb B' = \mathbb{B}_{j+2}$ of $n(v-2)$-dimensional cubes of side length $\ell$, whose cardinality obeys the inequality  
 \eqref{induction-j} with $j$ replaced by $j+1$, allowing us to carry the induction forward.  
\vskip0.1in 
\noindent We continue this contruction for $(v-1)$ steps, obtaining sets $S_1, \ldots, S_{v-1}$ and a collection of $\mathbb{B}_{v}$ consisting of at most $C M^{(n+1)(v-1)} N^{n(v-1)} \ell^{m - n}$ cubes of sidelength $\ell$ and dimension $n$ contained in $T_v$. 
The set $S_v$ is then defined according to the prescription of Lemma \ref{proj-lemma-v}, the conclusion of which verifies part \eqref{f-nonzero} of the proposition for $S_1, \cdots, S_v$. 
\end{proof}

\vskip0.1in 
\subsubsection{Projections of bad boxes} It remains to justify the projection mechanism used repeatedly in Proposition \ref{non-simul-mainprop vector-valued}. We set this up below. 
\vskip0.1in
\noindent Fix $2 \leq r \leq v$, and consider sets $T \subseteq [0,1]^n$ and $T' \subseteq [0,1]^{n(r-1)}$ expressible as unions of axis-parallel cubes of sidelength $M^{-1}$. As before, we denote by $\mathcal J_M(T)$ and $\mathcal J_M(T')$ the respective collections of these cubes. Given any $J \in \mathcal J_M(T)$, we decompose $J$ into subcubes of sidelength $N^{-1}$; the corresponding collection is termed $\mathcal I_N(J)$.  We will also need to fix a subset $B \subseteq T \times T'$, which we view as a union of a collection $\mathbb B$ of cubes of sidelength $\ell$. Here $M, N$ and $\ell$ are as specified in Proposition \ref{non-simul-mainprop vector-valued}.       
\begin{lemma}\label{proj-lemma}
Given $T, T', B$ as above, there exist sets $S \subseteq T$, $B' \subseteq T'$ and a collection of boxes $\mathbb B' \subseteq \mathbb T'$ with the following properties:
\vskip0.1in 
\begin{enumerate}[(a)]
\item \label{S-construction} The set $S$ is a union of cubes of sidelength $\ell$. More precisely, for every $J \in \mathcal J_M(T)$, there exists $\mathcal I_N'(J) \subseteq \mathcal I_N(J)$ such that \[\#(\mathcal I_N'(J)) \geq (1 - M^{-1}) \#(\mathcal I_N(J)),\] and $S \cap I$ is a single $\ell$-cube for each $I \in \mathcal I_N'(J)$. For $I \in \mathcal I_N(J) \setminus \mathcal I_N'(J)$, the set $S \cap I$ is empty. 
\vskip0.1in   
\item \label{B'} The set $B'$ is the union of the $\ell$-cubes in $\mathbb B'$.
\vskip0.1in
\item \label{cardinality-B'} $\#(\mathbb B') \leq M^{n+1} N^n \ell^n \#(\mathbb B)$.
\vskip0.1in
\item \label{set-containment} $(S \times T') \cap B \subseteq S \times B'$. 
\end{enumerate}   
\end{lemma}

\begin{proof}
Fix $J \in \mathcal J_M(T)$. For $I \in \mathcal I_N(J)$, define a ``slab''
\[W_N[I] := \bigcup \{Q = I \times I' \subseteq T \times T' : Q \text{ is a cube of sidelength } N{^-1} \}. \]
Thus a slab is the union of all of the axis-parallel boxes in $T \times T'$ of side length $\frac{1}{N}$ whose projection onto the $x_1$-coordinate is the cube $I$. Similarly, given an $n$-dimensional cube $I$ of sidelength $\ell$, we define a ``wafer'' $W_{\ell^{-1}}[I]$ to be the union of all cubes of sidelngth $\ell$ which project onto $I$ in the $x_1$-space. Let us observe that a slab is the disjoint union of exactly $N^{-n} \ell^{-n}$ wafers, and that the total number of wafers supported by $J$ is $M^{-n}\ell^{-n}$. A wafer in turn is a union of $\ell$-cubes. 
\vskip0.1in
\noindent Let us agree to call a wafer $W_{\ell^{-1}}[I^{\prime}]$ ``good'' if it contains at most $M^{n+1} \ell^n \#(\mathbb{B})$ boxes of $\mathbb{B}$. 
The pigeonhole principle dictates that the proportion of bad wafers is $ \leq \frac{1}{M}$. We will call a slab $W_N[I]$ ``good'' if it contains at least one good wafer. Again pigeonholing implies that no more than $\frac{1}{M}$-fraction of the slabs can be bad. 
Let us define $\mathcal I_N'(J)$ as the collection all cubes $I \in \mathcal I_N(J)$ such that $W_N[I]$ is good. For each cube $I \in \mathcal I_N'(J)$, we select one cube $I_0 = I_0(I) \subset I$ of sidelength $\ell$ such that $W_{\ell^{-1}}[I_0]$ is a good wafer. The set $S$ is now defined to be the union of all selected $\ell$-cubes $I_0(I)$, with $I \in \mathcal I_N'(J)$ and $J \in \mathcal J_M(T)$. Clearly, $S$ satisfies part (\ref{S-construction}) of the lemma. 
\vskip0.1in 
\noindent Let $B'$ be the union of the collection $\mathbb B'$ of all $\ell$-cubes $Q^{\prime} \subseteq T'$ such that $Q \times Q^{\prime} \in \mathbb{B}$ for some $\ell$-cube $Q \subseteq S$. Then (\ref{B'}) and (\ref{set-containment}) hold by definition. The selection algorithm for $S$ gives that for a given cube $Q \subseteq S$, the number of $Q^{\prime}$ such that $Q \times Q^{\prime} \in \mathbb{B}$ is  $\leq M^{n+1} \ell^n \# (\mathbb{B})$. On the other hand, each $Q \subseteq S$ comes from a distinct slab. Hence the total number of possible choices for $Q \subseteq S$ is no more than the total number of slabs, namely $N^n$. Combining all of this we get (\ref{cardinality-B'}) as desired.
\end{proof}
\vskip0.1in
\noindent A version of the lemma above is needed for the extreme case $r=1$. We needed this in the final step of the iterative process described in Proposition \ref{non-simul-mainprop vector-valued}, specifically in the construction of $S_v$. 
\begin{lemma} \label{proj-lemma-v} 
Let $T \subseteq [0,1]^n$ be the union of axis-parallel cubes of sidelength $M^{-1}$ and $B \subseteq T$ a union of such cubes of sidelength $\ell$. Decompose $T$ and $B$ into cubes of sidelength $N^{-1}$ and $\ell \ll N^{-1}$ respectively, denoting the corresponding collections $\mathbb T$ and $\mathbb B$. Suppose that 
\[ \#(\mathbb B) \leq C M^{(n+1)(v-1)} N^{n(v-1)} \ell^{m-n}, \quad \text{ with } \quad \ell \leq C^{-\frac{1}{m}} M^{-\frac{1}{m}\left( (n+1)v + 1\right)} N^{-\frac{n(v-1)}{m}}. \] Then there exist $S \subseteq T$ and $\mathbb T^{\ast} \subseteq \mathbb T$ such that 
\vskip0.1in 
\begin{enumerate}[(a)]
\item $S \cap B = \emptyset$. 
\vskip0.1in
\item $\#(\mathbb T^{\ast}) \geq (1 - 1/M) \#(\mathbb T)$. 
\vskip0.1in
\item $S$ is a union of a large number of $\ell$-cubes coming from $\mathbb T^{\ast}$. More precisely, $|S \cap I| \geq (1 - M^{-1}) N^{-n}$ for each $I \in \mathbb T^{\ast}$.   
\end{enumerate} 
\end{lemma}  
\begin{proof}
Decomposing each cube $I \in \mathbb T$ into subcubes of sidelength $\ell$, we declare $I$ to be good if it contains $\leq M^{n+1} N^{-n} \#(\mathbb B)$ subcubes that are in $\mathbb B$. As in the proof of Lemma \ref{proj-lemma}, the pigeonhole principle ensures that the fraction of bad cubes in $\mathbb T$ is at most $M^{-1}$. Define $\mathbb T^{\ast}$ to be the collection of good cubes in $\mathbb T$, and $S$ to be the collection of all subcubes of sidelength $\ell$ that are contained in the cubes of $\mathbb T^{\ast}$ but are disjoint from $B$. The relation between $\ell$, $M$ and $N$ implies that for every $I \in \mathbb T^{\ast}$, 
\[ |I \cap B| \leq M^{n+1} N^{-n}n\#(\mathbb B_v) \ell^n \leq C M^{(n+1)v} N^{n(v-2)} \ell^m \leq M^{-1} N^{-n}, \]
which justifies the size conclusion for $S$.    
\end{proof} 
\section{Proof of Theorems 1.1 and 1.2} \label{Proofs of non-simul thms} 
\noindent We present the construction of the set $E$ in Theorem \ref{non-simul-thm} in complete detail. The construction for Theorem \ref{non-simul-thm vector-valued} is similar. The small variations needed for this have been discussed in subsection \ref{construction-mods}.
\subsection{A sequence of differential operators} 
We will need to define a sequence of privileged derivatives in order to prove Theorem 1.1. For $\eta$ and $r_q$ as in the statement of Theorem \ref{non-simul-thm}, let $\alpha_q$ be a $v$-dimensional multi-index with $|\alpha_q| = r_q$ such that $\partial^{\alpha_q}f_q/\partial x^{\alpha_q}$ is nonvanishing everywhere on $[0, \eta]$. Here $\partial^{\beta}/\partial x^{\beta}$ denotes, following standard convention, the differential operator $\partial^{\beta_1+ \cdots + \beta_v}/\partial x_1^{\beta_1} \cdots \partial x_v^{\beta_v}$ of order $|\beta| = \beta_1 + \cdots = \beta_v$, if $\beta = (\beta_1, \cdots, \beta_v)$. We now define for each $q$ a finite sequence of privileged differential operators of diminishing order 
\begin{equation}  \mathcal D_{q}^{k} = \frac{\partial^{\alpha_{qk}}}{\partial x^{\alpha_{qk}}}, \qquad 0 \leq k \leq r_q.  \label{diffop}\end{equation} 
Here $\alpha_{q r_q} = \alpha_q$, and $\alpha_{q,k-1}$ is obtained by reducing the largest entry of $\alpha_{qk}$ by 1 and leaving the others unchanged. If there are multiple entries of $\alpha_{qk}$ with the largest value, we pick any one. Clearly $|\alpha_{qk}| = k$.      
\subsection{Construction of $E$}
The construction is of Cantor type with a certain memory-retaining feature inspired by the constructions
of Keleti \cite{{K98},{K08}}. This distinctive feature is the existence of an accompanying
queue that is, on one hand, generated by the construction and on the
other, contributes to it. More precisely, the $j$-th iteration of the construction is
predicated on the $j$-th member of the queue; at the same time the $j$-th step also
adds a large number of new members to the queue that become significant at
a later stage.
\subsubsection*{Step 0:}
At the initializing step, we set for $k = 1, \cdots, v$, 
\[ I_k[0] = \left[ (k-1) \frac{\eta}{v}, \frac{k \eta}{v} \right], \qquad \mathcal{E}_0 = \{ I_1[0], \ldots, I_{v}[0] \}, \qquad M_0 = \frac{v}{\eta}. \]
Letting $\Sigma_0$ denote the collection of injective mappings from $\{1, \ldots, v-1\}$ into $\{1, \ldots, v\}$, we define an ordered queue
\begin{align*} \mathcal{Q}_0 &= \{(1, m, \mathbb{I}_{\sigma}[0]): 0 \leq m \leq r_1 - 1, \sigma \in \Sigma_0\}, \quad \text{ where }  \\ 
\mathbb I_{\sigma}[0] &= (I_{\sigma(1)}[0], \ldots, I_{\sigma(v-1)}[0]).
\end{align*} 
The ordering in $\mathcal{Q}_0$ is as follows: Viewing $\Sigma_0$ as a collection of $(v-1)$-tuples with values from $\{1, \cdots, v \}$, we first endow $\Sigma_0$ with the lexicographic ordering, writing $\Sigma_0 = \{ \sigma_1 < \sigma_2 < \ldots \}$. Then $(1, m, \mathbb{I}_{\sigma_r}[0])$ precedes $(1, m^{\prime}, \mathbb{I}_{\sigma_{r'}}[0])$ in the list $\mathcal{Q}_0$ if one of the following scenarios holds: (a) $r < r'$ or (b) $r = r'$ and $m > m'$. 
\subsubsection*{Step 1:} Consider the first member of $\mathcal{Q}_0$, which is $(1, r_1 -1, \mathbb{I}_{\sigma_1}[0])$. Recalling the definition \eqref{diffop},  we observe that the hypotheses of Proposition \ref{non-simul-mainprop} are verified by \[ f = D^{r_1 - 1}_1 f_1, \quad (T_i : i \ne i_0) = \mathbb I_{\sigma_1}[0], \quad M = M_0. \] 
Here $i_0 = i_0(1)$ is the unique index in $\{1, 2, \cdots, v \}$ such that $\frac{\partial f}{\partial x_{i_0}} = \mathcal D_1^{r_1}f_1$, which is nonzero on $[0, \eta]$. The set $T_{i_0}$ will be the complement in $[0, 1]$ of $\cup_i \{ T_i : i \ne i_0\}$. The conclusion of Proposition \ref{non-simul-mainprop} therefore holds for some small constant $d_0 = c_1(M_0, \mathbb{T}) > 0$ and all sufficiently large integers $N_1$. We choose $N_1$ large enough so that $N_1 > e^{M_0}$, and sufficiently large that $N_1$ is greater than the value $M_0$ required for the application of Proposition \ref{non-simul-mainprop} for the function $f$ represented by the second queue element. In particular, Proposition \ref{non-simul-mainprop} ensures the existence of subsets $S_j \subset T_j$ for $1 \leq j \leq v$, each of which is a union of intervals of length $\ell_1 = \frac{d_0}{N_1^{v-1}}$ with 
\[ \mathcal D_1^{r_1 - 1} f_1(x) \neq 0 \text{ for }  x = (x_1, \cdots, x_v) \in S_1 \times \cdots \times S_v.\]
These constitute the basic intervals for the first stage.
\vskip0.1in 
\noindent 
Let $\mathcal{E}_1 = \{I_1[1], I_2[1], \cdots I_{L_1}[1]\}$ be an enumeration of the first stage basic intervals, and $\Sigma_1$ the collection of injective mappings from $\{ 1, \cdots, v-1 \}$ to $\{1, \cdots, L_1 \}$. We view an element of $\Sigma_1$ as an ordered $(v-1)$-tuple of distinct indices from $\{1, \cdots, L_1 \}$. As before, $\Sigma_1$ is arranged lexicographically.  Set \[\mathcal{Q}_1^{\prime} = \{(q, k, \mathbb{I}_{\sigma}[1]); 1 \leq q\leq 2; \; 0 \leq k \leq r_{q} - 1; \; \sigma \in \Sigma_1\}, \]
with $\mathbb I_{\sigma}[1] = (I_{\sigma(1)}[1], \cdots, I_{\sigma(v-1)}[1])$. 
The list $\mathcal{Q}_1^{\prime}$ is assigned the following ordering: an element of the form $(q, k, \mathbb{I}_{\sigma}[1])$ will precede $(q^{\prime}, k^{\prime}, \mathbb{I}_{\sigma'}([1]))$ if one of the following conditions holds: (a) $\sigma < \sigma'$, or (b) $\sigma = \sigma'$, $q<q^{\prime}$ or (c) $\sigma = \sigma'$,  $q=q^{\prime}$ and $k > k^{\prime}$.  The list $\mathcal{Q}_1^{\prime}$ is appended to $\mathcal{Q}_0$ to arrive at the updated queue $\mathcal{Q}_1$ at the end of step 1. 
\subsubsection*{The general step:}
In general, at the end of step $j$, we have the following quantities:
\begin{enumerate}[-] 
\item The $j$th iterate of the construction $E_j$, which is the union of the $j$th-level basic intervals of length $\ell_j = d_{j-1}/N_j^{v-1}$. Here $d_j$ is a sequence of small constants obtained from repeated applications of Proposition \ref{non-simul-mainprop} and depending on the collection of functions $\{f_q : q \leq j + 1 \}$. In particular, $d_{j-1}$ only depends on parameters involved in the first $(j-1)$ steps of the construction. The sequence $N_j$ is chosen to be rapidly increasing. For instance, choosing 
\begin{equation} \label{rapid-increase}
N_{j+1} > \exp \Biggl[ \prod_{k=1}^j \Bigl(\frac{N_k}{d_k} \Bigr)^{R}\Biggr] \quad \text{ for all } j \geq 1 
\end{equation} 
and some fixed large constant $R = R(v,n,m)$ would suffice. 
\vskip0.1in
\item The collection of the $j$th level basic intervals that constitute $E_j$, which we denote by $\mathcal E_j = \{I_1[j], I_2[j], \cdots, I_{L_j}[j] \}$. 
\vskip0.1in
\item The updated queue $\mathcal{Q}_j = \mathcal Q_{j-1} \cup \mathcal Q_j'$, with \[ \mathcal Q_j' = \{ (q,k, \mathbb I_{\sigma}[j]) : \; 1 \leq q\leq j+1, \; 0 \leq k \leq r_{q} - 1, \; \sigma \in \Sigma_j \}. \] Here $\Sigma_j$ is the collection of all injective maps from $\{1, \ldots, v-1\}$ to $\{1, \ldots, L_j\}$, which is viewed as the collection of all $(v-1)$-dimensional vectors with distinct entries taking values in $\{1, \cdots, L_j \}$ and endowed with the lexicographical order. The new list $\mathcal{Q}_j'$ is ordered in the same way as described in step 1 and appended to $\mathcal Q_{j-1}$. Notice that the number of members in the list $\mathcal Q_j$ is much larger than $j$. 
\end{enumerate} 
\vskip0.1in
\noindent We also know that $\mathcal{D}_q^k f_q(x)$ is nonzero for certain choices of $k, q,$ and $x$ with $k \leq r_q - 1$. Specifically, given any tuple of the form $(q, k, \mathbb I)$ that appears among the first $j$ members of the list $\mathcal{Q}_j$, the construction yields that 
\begin{equation}  |\mathcal D_q^k f_{q}(x)|  > 0 \text { if }  x_i \in E_j \cap I_i \text{ for } i \ne i_0, \; x_{i_0} \in E_j \setminus (I_1 \cup \cdots \cup I_{v-1}). \label{nonvanishing-derivs}  \end{equation} 
Here $i_0$ is the distinguished index such that $\partial \mathcal D^k_{q}f_q/ \partial x_{i_0} = \mathcal D_q^{k+1} f_q$. The $(v-1)$-tuple of intervals $\mathbb I$ has been labeled as $\mathbb I = (I_i : i \ne i_0)$. 
\vskip0.1in 
\noindent At step $(j+1)$, we refer to the $(j+1)$st entry of the queue $\mathcal Q_j$, which we denote by $(q_{0}, k_{0}, \mathbb I)$. Two cases can occur, depending on whether $k_{0}$ is maximal for the given $q_{0}$ or not. If it is, that means $k_{0} = r_{q_{0}} - 1$ for some $1 \leq q_{0} \leq j+1$. We want to apply Proposition \ref{non-simul-mainprop} with $M^{-1} =\ell_j$, 
\begin{equation}  
f = \mathcal D_{q_0}^{r_{q_0} - 1} f_{q_0},  \quad T_i = \begin{cases} 
E_j \cap I_i &\text{ if } i \ne i_0, \\  
E_j \setminus \bigcup_{i \ne i_0} T_i &\text{ if } i = i_0. \end{cases} \label{T-def} \end{equation}  
In this case, the nonvanishing derivative condition required for the application of Proposition \ref{non-simul-mainprop} is ensured by the hypothesis of Theorem \ref{non-simul-thm}. 
\vskip0.1in
\noindent The other possibility is when $k_{0} < r_{q_{0}}-1$. Given the specified ordering on $\mathcal{Q}_j$, we conclude that $(q_{0}, k_{0} + 1, \mathbb I)$ must be the $j$th member of $\mathcal{Q}_j$, and hence, by the induction hypothesis, \eqref{nonvanishing-derivs} holds with $q=q_0$ and $k = k_0+1$.  
We can now apply Proposition \ref{non-simul-mainprop} with $f = \mathcal D^{k_0}_{q_0} f_{q_{0}}$, $M^{-1} = \ell_j$, and the same choices of $i_0$ and $T_1, \cdots, T_v$ as in \eqref{T-def} above.
\vskip0.1in 
\noindent In either case, we obtain a collection $\mathcal{E}_{j+1}$ of $(j+1)$th level basic cubes of length $\ell_{j+1} = d_{j+1}/N_{j+1}^{v-1}$, the union of which is $E_{j+1}$, and for which \eqref{nonvanishing-derivs} holds with $q = q_0, k = k_0$ and $j$ replaced by $(j+1)$. This completes the induction. 
\subsection{Modifications to the construction of $E$ for Theorem \ref{non-simul-thm vector-valued}}\label{construction-mods}
The main distinction for Theorem \ref{non-simul-thm vector-valued} is that we only need to consider the first derivative $Df_q$ of $f_q$, so there is no need for the higher-order differential operators $\mathcal D_q^k$. What this means is that the elements of the queue $\mathcal{Q}_j^{\prime}$  are of the form $(q, \mathbb {I}_{\sigma}[j] )$, where $q$ ranges from $1$ to $j$ and $\mathbb{I}_{\sigma}$ is a tuple of cubes instead of intervals, and one needs to appeal to Proposition \ref{non-simul-mainprop vector-valued} instead of Proposition \ref{non-simul-mainprop}. The number of sub-cubes of $[0,\eta]^{nv}$ at the initializing step needs to be chosen large enough, so that their sidelengths do not exceed $M_0^{-1}$, as specified in the hypotheses of Proposition \ref{non-simul-mainprop vector-valued}. This is simply to ensure that Proposition \ref{non-simul-mainprop vector-valued} is applicable. 
%
%
\vskip0.1in
\noindent From this point forward, no distinction will be made between Theorem \ref{non-simul-thm} and the $m=1, n=1$ case of \ref{non-simul-thm vector-valued}. The computation of the Hausdorff and Minkowski dimensions of the set $E$ in these two cases proceeds in exactly the same manner.
\subsection{Nonexistence of solutions}
Fix any $q \geq 1$, and a tuple $x = (x_1, \cdots, x_v)$ of distinct points in $E$. Then there exists a step $j \geq q$ in the construction of $E$ where they lie in distinct basic intervals (in the case of Theorem \ref{non-simul-thm}) or cubes (in the case of Theorem \ref{non-simul-thm vector-valued}) of that step. Suppose that $\mathbb I^{\ast} = (I_1^{\ast}, \cdots, I_{v-1}^{\ast})$ is the tuple of $j$-th stage basic intervals such that $x_i \in I_i^{\ast}$. Set. Then the tuple $(q, 0, \mathbb I^{\ast})$ (or $(q, \mathbb{I}^{\ast})$ in the case of Theorem $1.2$) belongs to the list $\mathcal Q_j$. Suppose that it is the $j_0$th member of $\mathcal Q_j$, $j_0 \gg j$. This tuple then plays a decisive role at the $j_0$th step of the construction, at the end of which we obtain (either from Proposition \ref{non-simul-mainprop} or \ref{non-simul-mainprop vector-valued}) that $f_q$ does not vanish on $\prod_{i=1}^{v} E_{j_0} \cap I_i^{\ast}$.  Since $x$ lies in this set, we are done. 
\subsection{Hausdorff dimension of $E$}
Frostman's lemma dictates that the Hausdorff dimension of a Borel set $E$ is the supremum value of $\alpha > 0$ for which one can find a probability measure supported on $E$ with $\sup_{x,r} \mu(B(x;r))/r^{\alpha} < \infty$, where $B(x;r)$ denotes a ball centred at $x$ of radius $r$. Keeping in mind that any ball is coverable by a fixed number of cubes, we aim to construct a probability measure $\mu$ on $E$ with the property that for every $\epsilon > 0$, there exists $C_{\epsilon} > 0$ such that
\begin{equation} \mu(I) \leq C_{\epsilon} l(I)^{\frac{m}{v-1} - \epsilon} \text{ for all cubes $I$}. \label{Frostman-condition} \end{equation} 
Here $l(I)$ denotes the sidelength of $I$. 
\vskip0.1in
\noindent Let us recall that $\mathcal{E}_j$ denotes the collection of all basic cubes with sidelength $\ell_j$ at step $j$ of the construction. Decomposing each cube in $\mathcal{E}_j$ into equal subcubes of length $1/N_{j+1}$, we denote by $\mathcal{F}_{j+1}$ the resulting collection of subcubes that contain a cube from $\mathcal{E}_{j+1}$.  Let $F_{j+1}$ be the union of the cubes in $\mathcal{F}_{j+1}$. We define a sequence of measures $\nu_{j+1}$ and $\mu_j$ supported respectively on $F_{j+1}$
and $E_j$ as follows. The measure $\mu_0$ is the uniform measure on $[0,1]^n$.  Given $\mu_j$, the measure $\nu_{j + 1}$ will be supported on $F_{j+1}$ and will be defined by evenly splitting the measure $\mu_j$ of each cube in $\mathcal{E}_j$ among its children in $\mathcal{F}_{j+1}$. Given $\nu_j$, the measure $\mu_j$ will be supported on $E_j$ and will be defined by evenly splitting the measure $\nu_j$ of each cube in $\mathcal{F}_j$ among its children in $\mathcal{E}_j$. It follows from the mass distribution principle that the measures $\mu_j$ have a weak limit $\mu$. We claim that $\mu$ obeys the desired requirement \eqref{Frostman-condition}.
\vskip0.1in
\noindent The proof of the claim rests on the following proposition, which describes the mass distribution on the basic cubes of the construction.
\begin{proposition} \label{Hausdorff-dim-prop}
Let $K \in \mathcal{E}_j$, $J \in \mathcal{F}_{j+1}$ with $J \subset K$. Then
\begin{enumerate}[(a)]
\item \[\mu(K)/|K| \leq \mu(J)/|J| \leq 2 \mu(K)/|K|.\] \label{mJ/J} 
\item \[\mu(J) \leq M_j |J|, \text{ where } M_j = \prod_{k=1}^j 2(\ell_j N_j)^{-n}. \] \label{mJ} 
\end{enumerate}
\end{proposition}
\begin{proof}
We first prove part (a). Each $K \in \mathcal E_j$ decomposes into $(\ell_l N_{j+1})^n$ subcubes of sidelength $1/N_{j+1}$. Propositions \ref{non-simul-mainprop} and \ref{non-simul-mainprop vector-valued}  assert that at least a $(1 - 1/M)$-fraction of these subcubes contain a cube from $\mathcal E_{j+1}$ and hence lies in $\mathcal F_{j+1}$. The number of descendants $J \in \mathcal{F}_{j+1}$ of a given cube $K \in \mathcal{E}_j$ is therefore at most $(\ell_j N_{j+1})^n = |K|/|J|$ and at least $(\ell_j N_{j+1})^n/2 = |K|/(2|J|)$. Since $\mu(K)$ is evenly distributed among such $J$, part (\ref{mJ/J}) follows. 
\vskip0.1in
\noindent We prove part (\ref{mJ}) by applying part (\ref{mJ/J}) iteratively. Suppose that $\bar J$ is the cube in $\mathcal{F}_j$ that contains $K$. Then
\[\frac{\mu(J)}{|J|} \leq 2 \frac{\mu(K)}{|K|}\leq 2 \frac{\mu(\bar J)}{K}
= \frac{2 |\bar J|}{|K|} \frac{\mu(\bar J)}{|\bar J|} = \frac{2}{(\ell_j N_j)^n} \frac{\mu(\bar J)}{|\bar J|}.\]
\end{proof}
\vskip0.1in
\noindent We are now ready to apply Proposition \ref{Hausdorff-dim-prop} to prove \eqref{Frostman-condition}. Suppose that $I$ is a cube with sidelength between $\ell_{j+1}$ and $\ell_j$. There are two possibilities: either $\frac{1}{N_{j+1}} \leq l(I) \leq \ell_j$ or $\ell_{j+1} \leq l(I) < \frac{1}{N_{j+1}}$. 
\vskip0.1in
\noindent In the first case $I$ can be covered by at most $C|I|N_{j+1}^n$ cubes of sidelength $1/N_{j+1}$, all of which could be in $\mathcal F_{j+1}$. If $J$ is a generic member of $\mathcal F_{j+1}$, we obtain from Proposition \ref{Hausdorff-dim-prop} that  
\begin{multline*} \mu(I) \leq C |I| N_{j+1}^n \mu(J) \leq C |I| N_{j+1}^n M_j |J| \leq C M_j |I| \\ \leq C \frac{2M_{j-1}}{(\ell_j N_j)^n} |I| \leq C {M_{j-1}}{d_{j-1}^{-\frac{m}{v-1}}} \ell_{j}^{\frac{m}{v-1} - n} |I|\leq C_{\epsilon} \ell_{j}^{\frac{m}{v-1} - n - \epsilon} |I| \leq C_{\epsilon} l(I)^{\frac{m}{v-1} - \epsilon}.     \end{multline*}  
Here the penultimate inequality follows from the rapid growth condition \eqref{rapid-increase} . 
\vskip0.1in
\noindent Let us turn to the complementary case, when $\ell_{j+1} \leq l(I) \leq N_{j+1}^{-1}$. If $\mu(I) > 0$, the cube $I$ intersects at least one cube $J$ in $\mathcal F_{j+1}$ in which case it is contained in the union of at most $(2n+1)$ cubes of the same dimension adjacent to it. Proposition \ref{Hausdorff-dim-prop} then yields that  
\begin{multline*} \mu(I) \leq C_n \mu(J) \leq C_n M_j |J| = C_n M_j N_{j+1}^{-n} \\ = C_n M_j d_j^{-\frac{m}{v-1}} \ell_{j+1}^{\frac{m}{v-1}} \leq C_{\epsilon} \ell_{j+1}^{\frac{m}{v-1} - \epsilon} \leq C_{\epsilon} l(I)^{\frac{m}{v-1} - \epsilon}, \end{multline*} 
applying \eqref{rapid-increase} as before at the penultimate stage. This establishes the claim \eqref{Frostman-condition}.  
\subsection{Minkowski dimension of $E$}
In order to establish the full Minkowski dimension of $E$, we show that for any $\epsilon > 0$, there exists $c_{\epsilon} > 0$ such that 
\begin{equation} \label{to-show-Minkowski-dim} 
\mathcal N_{\ell}(E) \geq c_{\epsilon} \ell^{-n+\epsilon} \text{ for any } 0 < \ell \ll 1.
\end{equation} 
Here $\mathcal N_{\ell} (E)$ denotes the smallest number of closed cubes of sidelength $\ell$ required to cover $E$. As before we study two cases, namely $\ell_{j+1} \leq \ell < 1/N_{j+1}$ and $1/N_{j+1} \leq \ell < \ell_j$. 
\vskip0.1in
\noindent If $\ell \in [\ell_{j+1}, 1/N_{j+1})$, we select $I \in \mathcal E_j$ of sidelength $\ell_j$ such that \begin{equation} \label{pickI} \begin{cases} I \subseteq T_{i_0(j+1)} \text{ for Theorem \ref{non-simul-thm}}, \\  I \subseteq T_{v}[j+1] \text{ for Theorem \ref{non-simul-thm vector-valued}}.\end{cases} \end{equation}  Here $i_0(j+1) \in \{1, \cdots, v\}$ denotes the preferred index at step $(j+1)$ of the construction, based on which Proposition \ref{non-simul-mainprop} is applied. On the other hand, $T_v[j+1]$ denotes the choice of $T_v$ at the $(j+1)$-th step for the purpose of applying Proposition \ref{non-simul-mainprop vector-valued}. In either case, $I \in \mathcal E_j$ can be partitioned into $(\ell_j N_{j+1})^n$ subcubes of sidelength $1/N_{j+1}$. It follows from \eqref{I_N'I_N} and \eqref{I_N'I_N vector-valued} in Propositions \ref{non-simul-mainprop} and \ref{non-simul-mainprop vector-valued} that  at least half of these subcubes lie in $\mathcal F_{j+1}$. Further, the conclusions \eqref{S_i0} and \eqref{vsize} of the propositions say that each $J \in \mathcal F_{j+1}$, \begin{equation}  \label{JE} |J \cap E_{j+1}| \geq \frac{N_{j+1}^{-n}}{2}.\end{equation} 
In view of the restriction $\ell \leq 1/N_{j+1}$ and \eqref{rapid-increase}, this leads to \begin{align*} \mathcal N_{\ell}(E) \geq \mathcal N_{\ell}(I \cap E) &\geq c \sum_J \left\{ \mathcal N_{\ell}(J \cap E) : J \subseteq I, \; J \in \mathcal F_{j+1} \right \} \\ &\geq c \sum_{J} \left\{ \frac{|J \cap E_{j+1}|}{\ell^n} : J \subseteq I, \; J \in \mathcal F_{j+1} \right\} \\ &\geq \frac{c}{2}(\ell_{j}N_{j+1})^n \times \frac{1}{2} \frac{N_{j+1}^{-n}}{\ell^n} = \frac{\ell_j^n}{4\ell^n} \geq c_{\epsilon} \ell^{n-\epsilon}. \end{align*}    
\vskip0.1in
\noindent Now let us consider the second case, where $\ell \in [1/N_{j+1}, \ell_j)$. The analysis is similar. Pick $I \in \mathcal E_{j-1}$ such that \eqref{pickI} holds with $j$ replaced by $(j-1)$.  As before, we decompose $I$ into cubes $J \in \mathcal F_j$, each of which obeys \eqref{JE}, also with $j$ replaced by $(j-1)$.
Since $\ell < \ell_j \leq 1/N_j$ an argument analogous to the one in the last paragraph leads to 
\begin{align*} \mathcal N_{\ell}(E) &\geq \mathcal N_{\ell}(E \cap I) \geq c \sum_{J} \left\{ \frac{|E_j \cap J|}{\ell^n}: J \subseteq I, \; J \in \mathcal F_j\right\} \\  &\geq \frac{c}{2}(\ell_{j-1} N_j)^n \times \frac{N_j^{-n}}{2 \ell^n} = \frac{\ell_{j-1}^n}{4\ell^n} \geq c_{\epsilon} \ell^{n-\epsilon},
\end{align*} 
with the last step using \eqref{rapid-increase} and the bounds on $\ell$.  This completes the proof. 
\section{Zero sets of functions with a common linearization} \label{Proof of simul-thm}
\noindent We now turn our attention to the proof of Theorem \ref{simul-thm}. Not surprisingly in view of the other results in this paper, it is also predicated on an iterative algorithm which has been encapsulated in Proposition \ref{mainprop-simul} below. The following lemma provides a preparatory step. 
\vskip0.1in 
\noindent  Let $\alpha \in \mathbb R^v$ be as in the statement of Theorem \ref{simul-thm}, and let $\mathfrak C$ be a nonempty strict subset of the index set $\{1, 2, \cdots, v\}$. Let $\delta > 0$. Consider disjoint intervals $[a_1, b_1]$ and $[a_2, b_2]$ of length $\lambda$, with $a_1 < b_1 < a_2 < b_2$. We define two quantities $\epsilon_{\text{left}}$ and $\epsilon_{\text{right}}$ depending on $\mathfrak C, a_1, b_1, a_2, b_2$ and $\delta$ as follows: 
\begin{align} 
\epsilon_{\text{left}} &:= \sup \left\{\epsilon: \bigl| \sum_{j=1}^{v} \alpha_j z_j \bigr| \geq \delta \lambda \text{ for }  \begin{cases} z_j \in [a_1, a_1 + \epsilon \lambda] &\text{ for all } j \notin \mathfrak{C} \\ z_j \in [a_2, a_2 + \epsilon \lambda] &\text{ for all }  j \in \mathfrak{C}. \end{cases} \right\} \label{alpha-def}\\
\epsilon_{\text{right}} &:= \sup \left\{\epsilon: \bigl| \sum_{j=1}^{v} \alpha_j z_j \bigr| \geq \delta \lambda \text{ for }  \begin{cases} z_j \in [a_1, a_1 + \epsilon \lambda] &\text{ for all } j \notin \mathfrak{C} \\ z_j \in [b_2 - \epsilon \lambda, b_2] &\text{ for all }  j \in \mathfrak{C}. \end{cases} \right\} \label{beta-def} 
\end{align} 
\begin{lemma} \label{simul-lemma} 
Given any $\alpha \in \mathbb R^v$ as in Theorem \ref{simul-thm}, there exists $\delta_0 > 0$ depending only on $\alpha$ such that for any $\lambda > 0$ and any choice of intervals $\mathfrak I_1 = [a_1, b_1]$ and $\mathfrak I_2 = [a_2, b_2]$ of equal length $\lambda$ with $a_1 < b_1 \leq a_2 < b_2$, the following property holds. For any $\delta < \delta_0$, there exists $\epsilon_0 = \epsilon_0(\mathfrak C, \delta)$ (not depending on $a_1, a_2, b_1, b_2,$ or $\lambda$) such that $\max(\epsilon_{\text{left}}, \epsilon_{\text{right}}) \geq \epsilon_0$. 
\vskip0.1in
\noindent In particular, there exist subintervals $\widehat{\mathfrak I}_1 \subseteq  \mathfrak I_1$ and $\widehat{\mathfrak I}_2 \subseteq \mathfrak I_2$ with $|\widehat{\mathfrak I}_1| = |\widehat{\mathfrak I}_2| = \epsilon_0 \lambda$ and dist$(\widehat{\mathfrak I}_1, \widehat{\mathfrak I}_2) \geq (1 - \epsilon_0) \lambda$ such that \[|\alpha \cdot x| \geq \delta \lambda \text{ for all $x \in \mathbb R^v$ such that } \begin{cases} x_j \in \widehat{\mathfrak I}_1 &\text{ for } j \not\in \mathfrak C, \\ x_j \in \widehat{\mathfrak I}_2 &\text{ for } j \in \mathfrak C.  \end{cases}  \]  
\end{lemma}
\begin{proof}
Set $g(y) = \sum_j \alpha_j y_j$, and consider $g(z^{\ast})$, where $z^{\ast} = (z_1^{\ast}, \cdots, z_v^{\ast})$ is defined to be the $v$-dimensional vector with $z_j^{\ast} = a_1$ if $j \notin \mathfrak{C}$ and $z_j^{\ast} = a_2$ if $j \in \mathfrak{C}$. Setting $C^{\ast} = \sum_j |\alpha_j|$, we note that \begin{equation} \label{2g}  |g(z) - g(z^{\ast})| \leq C^{\ast} \epsilon \lambda \quad \text{ whenever } |z_j - z_j^{\ast}| \leq \epsilon \lambda, \; 1 \leq j \leq v. \end{equation} 
\vskip0.1in
\noindent If $|g(z^\ast)| > (\delta + \epsilon_0 C^{\ast})\lambda$, then \eqref{2g} implies that $|g(z)| \geq \delta \lambda$ for any $z$ as in \eqref{alpha-def}. Therefore $\epsilon_{\text{left}} \geq \epsilon_0$, and the conclusion of the lemma holds with $\widehat{\mathfrak I}_1 = [a_1, a_1 + \epsilon_0 \lambda]$, $\widehat{\mathfrak I}_2 = [a_2, a_2 + \epsilon_0 \lambda]$. Otherwise, let $\widehat{z} = (\widehat{z}_1, \cdots, \widehat{z}_v)$ be the $v$-dimensional vector with $\widehat{z}_j = a_1$ if $j \notin \mathfrak C$ and $\widehat{z}_j = b_2$ if $j \in \mathfrak C$. Then $g(\widehat{z}) = g(z^{\ast}) + \alpha \cdot (\widehat{z} - z^{\ast}) = g(z^{\ast}) + (b_2-a_2)C_0 = g(z^{\ast}) + \lambda C_0$, where $C_0 = \bigl|\sum_{j \in \mathfrak{C}} \alpha_j\bigr| > 0$. Thus, for $z$ as in \eqref{beta-def}, we obtain the estimate 
\begin{multline*} |g(z)| \geq |g(\widehat{z})| - |\alpha \cdot (z - \widehat{z})| \geq |C_0 \lambda + g(z^{\ast})| - C^{\ast} \epsilon_0 \lambda \\ \geq C_0 \lambda - (\delta + C^{\ast} \epsilon_0) \lambda - C^{\ast} \epsilon_0 \lambda \geq C_0 \lambda - (\delta + 2 \epsilon_0 C^{\ast}) \lambda,\end{multline*} which is greater than or equal to $\delta \lambda$ provided that $\delta < C_0/2 =: \delta_0$ and $\epsilon_0 < (C_0 - 2 \delta)/(2C^{\ast})$. One has $\epsilon_{\text{right}} \geq \epsilon_0$ for this choice of $\epsilon_0$, with the conclusion of the lemma verified for $\widehat{\mathfrak I}_1 = [a_1, a_1 + \epsilon_0 \lambda]$, $\widehat{\mathfrak I}_2 = [b_2 - \epsilon_0 \lambda, b_2]$.
\end{proof}
\vskip0.1in 
\noindent {\em{Remarks: }} \begin{enumerate}[(a)]
\item Let us consider the example $\alpha = (1, -2, 1)$, which corresponds to a linear function $g$ that picks out three-term arithmetic progressions. Choose $\mathfrak{C}$ to be $\{3 \}$ . For $x_1, x_2 \in [a_1, a_1 + \epsilon \lambda]$ and $x_3 \in [a_2, a_2 + \epsilon \lambda]$, it is easy to see that \[ x_1 -2x_2 + x_3 \geq a_1 + a_2 - 2(a_1 + \epsilon \lambda) = a_2-a_1 - 2 \epsilon \lambda \geq (1 - 2 \epsilon) \lambda. \] We can therefore take $\epsilon_{\text{left}} = \frac{1 - \delta}{2}$. On the other hand, if $x_1, x_2 \in [a_1, a_1 + \epsilon \lambda]$ and $x_3 \in [b_2 - \epsilon \lambda, b_2]$, then 
\[ x_1 - 2x_2 + x_3 \geq a_1 + b_2 - \epsilon \lambda - 2(a_1 + \epsilon \lambda) = b_2 - a_1 - 3 \epsilon \lambda \geq (2 - 3 \epsilon) \lambda. \]
Thus $\epsilon_{\text{right}} = \frac{2-\delta}{3}$. 
The point is that, in the above lemma, it is possible in certain instances for both $\epsilon_{\text{left}}$ and $\epsilon_{\text{right}}$ to be bounded from below. The lemma guarantees that at least one of them will be.
\vskip0.1in 
\item It is important to be aware that the above proof does not necessarily give the best possible $\epsilon_0$ for a given $\delta$ because the signs of the components of $\alpha$ are not taken into account. When dealing with a specific $\alpha$, it is often possible to improve the bound on $\epsilon_0$ given above.
\end{enumerate}
\vskip0.1in 
\begin{proposition} 
Let $I$ be an interval of length $\ell$, and let $I_1$ and $I_2$ denote the two halves of $I$. Then for every sufficiently small $\delta > 0$ there exists $\epsilon(\delta) > 0$ and subintervals $I_1'$ and $I_2'$ of $I_1$ and $I_2$ of length $\epsilon \ell$, such that $|\alpha \cdot x| \geq \delta \ell$ for any choice of $x_1, x_2, \ldots, x_v \in I_1^{\prime} \cup I_2^{\prime}$, not all of which are in $I_i^{\prime}$ for a single $i=1,2$.  The subintervals $I_1'$ and $I_2'$ are separated by at least $\ell/4$. \label{mainprop-simul}
\end{proposition}
\begin{proof}
Let $\{\mathfrak C_1, \mathfrak C_2, \cdots, \mathfrak C_R\}$ be an enumeration of all nonempty, strict subsets of $\{1, 2, \cdots, v\}$. Given any $x = (x_1, \cdots, x_v)$ such that $x_j \in I$ for all $j$ but not all $x_j$-s lie in a single $I_1$ or $I_2$, there exists $1 \leq m \leq R$ such that $j \in \mathfrak C_m$ if and only if $x_j \in I_2$.  
\vskip0.1in
\noindent Starting with $I_1$ and $I_2$, we apply Lemma \ref{simul-lemma} with $\mathfrak C = \mathfrak C_1$, $\mathfrak I_1 = I_1$, $\mathfrak I_2 = I_2$ and $\lambda = \ell/2$. For $2 \delta \leq \delta_0$, this gives a constant $\epsilon_1 = \epsilon_0(\mathfrak C_1, 2\delta)>0$ and two subintervals $I_1^{(1)} \subseteq I_1$ and $I_2^{(1)} \subseteq I_2$ of length $\epsilon_1 \ell/2$ obeying the conclusions of the lemma. For $2 \leq k \leq R$, we continue to apply Lemma \ref{simul-lemma} recursively, with \[\mathfrak C = \mathfrak C_k,\;  \mathfrak I_1 = I_1^{(k-1)}, \; \mathfrak I_2 = I_2^{(k-1)}, \; \lambda = \epsilon_1 \cdots \epsilon_{k-1} \ell/2.\]  At the end of the $k$-th step, this yields a constant $\epsilon_k = \epsilon_0(\mathfrak C_k, 2 \delta)$ and subintervals $I_1^{(k)} \subseteq I_1^{(k-1)} \subseteq I_1$, $I_2^{(k)} \subseteq I_2^{(k-1)} \subseteq I_2$ each of length $\epsilon_1 \cdots \epsilon_k \ell/2$ such that for any $m \leq k$, 
\[ |\alpha \cdot x| \geq \delta \ell \text{ for all $x$ such that } \begin{cases} x_j \in I_1^{(k)} & \text{ for }  j \not\in \mathfrak C_m, \\ x_j \in I_2^{(k)} &\text{ for }  j \in \mathfrak C_m.  \end{cases}  \] The conclusion of the proposition then holds for $I_1' = I_1^{(R)}$, $I_2' = I_2^{(R)}$ and $\epsilon = (\prod_{k=1}^{R}\epsilon_k)/2$.  The separation condition is an easy consequence of the one in Lemma \ref{simul-lemma}, since $I_i' \subseteq I_i^{(1)}$ for $i=1,2$.  
%
\end{proof}
\vskip0.1in 
{\em{Remarks: }} \begin{enumerate}[(a)] 
\item Tracking the parameters from Lemma \ref{simul-lemma}, we find that the constant $\epsilon$ claimed in Proposition \ref{mainprop-simul} obeys the estimate \begin{equation}  \epsilon \geq \frac{1}{2} \prod_{m=1}^R \frac{(C_m - 2 \delta)}{(2C^{\ast})}, \quad \text{ where }  C_m = |\sum_{j \in \mathfrak C_m} \alpha_j|. \label{eps-lower-bound} \end{equation}  
\vskip0.1in 
\item In view of the remarks made at the end of Lemma \ref{simul-lemma}, it is not surprising that the bound on $\epsilon$ in the preceding inequality is not always optimal. Returning to the example $\alpha = (1, -2, 1)$, we leave the reader to verify that given any small $\delta > 0$ and $I = [a, a+\ell]$, the choice $I_1' = [a, a+(1- \delta) \ell/3]$ and $I_2' = [a + (2+\delta)\ell/3,a+\ell]$ meets the requirements of the proposition. Thus for this $\alpha$, the best choice of $\epsilon$ is at least $(1 - \delta)/3$, which is much better than the one provided by the proof.  \label{eps-nonoptimal} 
\end{enumerate} 
\subsection{Proof of Theorem \ref{simul-thm}} \label{simul-thm-proof}
\begin{proof}
Fix a constant $\delta > 0$ arbitrarily small, and recall that $g(x_1, \cdots, x_v) = \sum_{j=1}^{v} \alpha_j x_j$ 
Start with $E_0 = [0, \eta]$ where $0 < \eta \ll 1$ is chosen sufficiently small so as to ensure $2Kv \eta < \delta$. Applying Proposition \ref{mainprop-simul} with $I = E_0$, 
we arrive at subintervals $I_1' = J_1 \subseteq [0, \eta/2]$ and $I_2' = J_2 \subseteq [\eta/2, \eta]$ of length $\epsilon \eta$ that obey its conclusions. Let $E_1 = J_1\cup J_2$ with $|J_1| = |J_2| = \ell_1$. In general, if $E_j$ is a disjoint union of $2^j$ basic intervals of length $\ell_j = \epsilon^j \eta$, then at step $(j+1)$, we apply Proposition \ref{mainprop-simul} to each such interval to find two subintervals of length $\ell_{j+1} = \epsilon \ell_{j} = \epsilon^{j+1} \eta$ and separated by a length of at least $\ell_j/4$, which form the basic intervals of $E_{j+1}$.
\vskip0.1in
\noindent  Defining $E = \cap_{j=1}^{\infty} E_j$, we now show that $f(x_1, \cdots, x_v) \ne 0$ if $x_1, \cdots, x_v$ are not all identical and $f$ is of the form \eqref{f-form}. For any such choice of $x_1, \cdots, x_v$, there exists a largest index $j$ such that $x_1, x_1, \cdots, x_v$ all lie in a basic interval $I$ at step $j$. This means that if $I_1'$ and $I_2'$ are the two subintervals of $I$ generated by Proposition \ref{mainprop-simul}, then $x_1, \cdots, x_v$ lie in $I_1' \cup I_2'$, but not all of them lie in a single $I_i'$. If $I$ is of length $\ell_j$, it follows from Proposition \ref{mainprop-simul} that $|g(x)| \geq \delta \ell_j$. But $|f(x) - g(x)| \leq Kv \ell_j^2$ according to \eqref{G}, so this implies $|f(x)| \geq \frac{\delta \ell_j}{2}$ for $\ell_j < \eta.$
\vskip0.1in 
\noindent We recall that the $(j+1)$th step of the construction generates exactly two children from each parent, and these are separated by at least $\ell_j/4$. It now follows from standard results (see for instance \cite{F97}, Example 4.6, page 64)  that the Hausdorff dimension of $E$ is bounded from below by 
\[ \lim_{j \rightarrow \infty}  \frac{\log(2^{j})}{-\log( 2\ell_j/4)} = \lim_{j \rightarrow \infty}  \frac{\log(2^{j})}{-\log(\epsilon^j \eta/2)} = \frac{\log 2}{-\log \epsilon}. \]
This establishes the existence of the set claimed by the theorem, with $c(\alpha) = \log 2/\log(\frac{1}{\epsilon})$, where $\epsilon$ is at least as large as the bound given in \eqref{eps-lower-bound}. 
\end{proof} 
\vskip0.1in 
\noindent {\em{Remark: }} We return to the example $\alpha = (1, -2, 1)$ that we have been following across this section to show that the avoiding set in this instance can be chosen to have Hausdorff dimension $\log 2/\log 3$. We have referred to this fact in certain examples occurring in Sections \ref{example-all} and \ref{discussion-optimality}.  
\vskip0.1in
\noindent Choose a slowly decreasing sequence $\delta_j = 1/(j+C)$, for some fixed large constant $C$. We have seen, in the remark (\ref{eps-nonoptimal}) following Proposition \ref{mainprop-simul}, that $\epsilon(\delta_j) = \epsilon_j$ can be chosen as $(1-\delta_j)/3$. Let us now use the same Cantor construction as in the proof given above, but using the parameter $\delta_j$ at step $j$ instead of a fixed $\delta$. The following consequences are immediate: 
\begin{align*} 
\ell_j &= \epsilon_1 \cdots \epsilon_j \eta \quad \text{ so that } \quad \ell_j \leq \frac{C \eta 3^{-j}}{j+C} . \\ 
|g(x)| &\geq \delta_j \ell_j \text{ and } |f(x) - g(x)| \leq Kv \ell_j^2, \text{ so that } \\ |f(x)| &\geq (\delta_j - Kv \ell_j) \ell_j \geq \Bigl (\frac{1}{j+C} - \frac{Kv \eta C}{j+C}  3^{-j} \Bigr) \ell_j > 0, 
\end{align*} 
where $x= (x_1,\cdots, x_v)$ is as in the second paragraph of Section \ref{simul-thm-proof}. This proves the nonexistence of nontrivial zeros of $f$. Further, the Hausdorff dimension is bounded from below by 
\[ \lim_{j \rightarrow \infty} \frac{\log(2^j)}{- \log(2 \ell_j/4)} = \lim_{j \rightarrow \infty} \frac{\log(2^j)}{- \log\bigl( 3^{-j} \eta \prod_{k=1}^{j} (1 - \delta_k)/2 \bigr)} = \frac{\log 2}{\log 3},  \]  
establishing the claim. 

\section{Appendix} 
\noindent We collect here the proofs of a few technical facts mentioned in Section \ref{ex-section}. 
\begin{lemma} \label{d-differentiable}
Given a $C^2$ parametrization $\gamma: [0,\eta] \rightarrow \mathbb R^n$ of a curve $\Gamma$, let us recall the definition of the signed distance function $d$ from \eqref{def-d}. Set $F(t_1, t_2) = d(\gamma(t_1), \gamma(t_2))$. Then 
\begin{enumerate}[(a)]
\item $F$ is differentiable on $[0,\eta]^2$. 
\item If $\gamma$ is the arclength parametrization, i.e., $|\gamma'(t)|\equiv 1$, then \[ \frac{\partial F}{\partial t_1}(t,t) = 1, \qquad \frac{\partial F}{\partial t_2}(t,t) = -1. \]  \label{special-gamma}
\end{enumerate} 
\end{lemma}
\begin{proof} 
Since differentiability is obvious for $t_1 \ne t_2$, it suffices to verify it when $t_1 = t_2 = t$. We consider two cases. If $h \geq k$, then
\begin{IEEEeqnarray*}{rCl}
F(t+h, t+k) =  d(\gamma(t + h), \gamma(t + k)) & = & |\gamma(t + h) - \gamma(t + k)|\\
& = & |\gamma^{\prime}(t)| \abs{h - k} + O(h^2 + k^2)\\
& = & |\gamma^{\prime}(t)| (h - k) + O(h^2 + k^2).
\end{IEEEeqnarray*}
On the other hand if $h < k$, we have
\begin{IEEEeqnarray*}{rCl}
d(\gamma(t + h), \gamma(t + k)) & = & |\gamma(t + h) - \gamma(t + k)| \\
& = & - |\gamma^{\prime}(t)| \abs{h - k} + O(h^2 + k^2)\\
& = & |\gamma^{\prime}(t)| (h - k) + O(h^2 + k^2)
\end{IEEEeqnarray*}
This establishes the first part of the lemma, with \[ \frac{\partial F}{\partial t_1}(t,t) = |\gamma'(t)|, \qquad \frac{\partial F}{\partial t_2}(t,t) = - |\gamma'(t)|.\] The second part is now obvious.  
\end{proof} 
\begin{lemma}\label{order-lemma}
Let $\gamma:[0, \eta] \rightarrow \mathbb R^n$ be an injective parametrization of a $C^2$ curve with curvature at most $K$ and $\gamma'(0) \ne 0$. If $\eta$ is sufficiently small depending on $|\gamma^{\prime}(0)|$ and $K$, then there are no isosceles triangles $\gamma(t_1), \gamma(t_2), \gamma(t_3)$ with $0 \leq t_1 < t_2 < t_3 \leq \eta$ whose sides of equal length meet at $\gamma(t_1)$ or at $\gamma(t_3)$.
\end{lemma}
\begin{proof}
Since $d$ has already been shown to be differentiable in the previous lemma, we compute 
\begin{IEEEeqnarray}{rCl}
d(\gamma(t_3), \gamma(t_1)) - d(\gamma(t_2), \gamma(t_1)) & = & \int_{t_2}^{t_3} \frac{\partial}{\partial t} d(\gamma(t), \gamma(t_1)) \nonumber \\
 & = & \int_{t_2}^{t_3} \gamma^{\prime}(t) \cdot \frac{\gamma(t) - \gamma(t_1)}{|\gamma(t) - \gamma(t_1)|}. \label{int}
\end{IEEEeqnarray}
For $t, t_1 \in [0, \eta]$ with $t > t_1$, we obtain
\begin{align*} \frac{\gamma(t) - \gamma(t_1)}{|\gamma(t) - \gamma(t_1)|} &= \frac{\left[ \gamma'(t_1) (t-t_1) + O(K(t-t_1)^2)  \right]}{|\left[ \gamma'(t_1) (t-t_1) + O(K(t-t_1)^2)  \right]|} \\ &= \frac{\gamma'(t_1) + O(K \eta)}{|\gamma'(t_1) + O(K \eta)|} \\ &= \frac{\gamma'(0) + O(K \eta)}{|\gamma'(0) + O(K \eta)|} \\
 &= \frac{\gamma'(0)}{|\gamma'(0)|} \left[ 1+ O\left( \frac{K \eta}{|\gamma'(0)|}\right) \right].   
\end{align*} 
Using this, the integrand in \eqref{int} may be estimated as follows, 
\[ \gamma'(t) \cdot \frac{\gamma(t) - \gamma(t_1)}{|\gamma(t) - \gamma(t_1)|} = \left[\gamma'(0) + O(K\eta) \right] \cdot \frac{\gamma'(0)}{|\gamma'(0)|} \left[1 + O\left( \frac{K \eta}{|\gamma'(0)|}\right) \right] \geq \frac{|\gamma'(0)|}{2} \neq 0,   \]
provided $\eta$ is small relative to $K$ and $|\gamma'(0)|$.   
This shows that \[d(\gamma(t_3), \gamma(t_1)) - d(\gamma(t_2), \gamma(t_1)) \geq |\gamma'(0)| (t_3-t_2)/2 \ne 0,\] proving that $\gamma(t_1)$ cannot be the vertex at the intersection of two equal sides in an isosceles triangle. A similar argument works for $\gamma(t_3)$.
\end{proof}

\begin{lemma} \label{Jacobian-check} 
Given a curve $\Gamma$ as described in Section \ref{trapezoid-section}, let us recall the function $f = (f_1, f_2)$ given by \eqref{def-f1-trapezoid} and \eqref{def-f2-trapezoid}. Then $Df$ is of full rank whenever $f = 0$. 
\end{lemma} 
\begin{proof} 
To prove that $Df$ has rank  2 on the zero set of $f$, it suffices to show that the $2 \times 2$ submatrix with entries $\partial f_i/\partial t_j$ with $i=1,2$ and $j=1,4$ is nonsingular. We will do this by proving that $\partial f_1/\partial t_j$ are nonzero and of the same sign for $j = 1,4$, whereas for $\partial f_2/\partial t_j$ the signs are reversed. Let us observe that $f_1$ is nonzero if $t_4-t_1$ and $t_3-t_2$ have opposite signs, whereas $f_2$ is nonzero if $t_4-t_3$ and $t_2-t_1$ have opposite signs. In what follows, we will therefore restrict to the case where $(t_4-t_1)(t_3-t_2) > 0$ and $(t_4-t_3)(t_2-t_1) > 0$. 
\vskip0.1in  
\noindent We begin by computing ${\partial f_1}/{\partial t_1}$. Let $(x_1, y_1) = \gamma(t_1)$ and $(x_4, y_4) = \gamma(t_4)$. For $t_1 < t_4$, our assumptions on $\gamma$ dictate that $\gamma(t_4)  - \gamma(t_1)$ lies in the first quadrant, so that
\[
\text{arg}(\gamma(t_4) - \gamma(t_1)) =  \arctan \left( \frac{y_4 - y_1}{x_4 - x_1} \right); \text{ this leads to }   
\]
\begin{IEEEeqnarray*}{rCl}
\frac{\partial f_1}{\partial t_1} =  \frac{\partial}{\partial t_1} \left[\text{arg}( \gamma(t_4) - \gamma(t_1)) \right]& = & \frac{1}{1 + \left( \frac{y_4 - y_1}{x_4 - x_1} \right)^2} \frac{- (x_4 - x_1) y^{\prime}(t_1) + (y_4 - y_1)x^{\prime}(t_1)}{(x_4 - x_1)^2} \\
& = & \frac{1}{1 + \left(\frac{y_4 - y_1}{x_4 - x_1}\right)^2} \frac{(y_4 - y_1)x^{\prime}(t_1) - (x_4 - x_1)y^{\prime}(t_1)}{(x_4 - x_1)^2}.
\end{IEEEeqnarray*}
For $t_1 < t_4$, the above expression is positive for a strictly convex curve $\Gamma$ of the type we have assumed. For $t_1 > t_4$ the sign is reversed.  
It is easily seen that ${\partial f_1}/{\partial t_4}$ essentially has the same expression as ${\partial f_1}/{\partial t_1}$, but the numerator has opposite sign and $x^{\prime}$ and $y^{\prime}$ are evaluated at $t_4$ instead of at $t_1$. 
Strict convexity of $\Gamma$ therefore implies that ${\partial f_1}/{\partial t_1}$ and ${\partial f_1}/{\partial t_4}$ have the same sign.
\vskip0.1in
\noindent Now let us consider $\partial f_2/\partial t_j$ for $j=1,4$. We find that 
\begin{align*} \frac{\partial}{\partial t_4}d(\gamma(t_4), \gamma(t_3))  &= \gamma^{\prime}(t_4) \cdot \frac{\gamma(t_4) - \gamma(t_3)}{|\gamma(t_4) - \gamma(t_3)|}, \; \text{ so } \\ \frac{\partial f_2}{\partial t_4} &= \gamma^{\prime}(t_4) \cdot \frac{\gamma(t_4) - \gamma(t_3)}{|\gamma(t_4) - \gamma(t_3)|} d(\gamma(t_2), \gamma(t_1)). \\
\text{Similarly } \frac{\partial f_2}{\partial t_1} &= - \gamma^{\prime}(t_1) \cdot \frac{\gamma(t_2) - \gamma(t_1)}{|\gamma(t_2) - \gamma(t_1)|} d(\gamma(t_4), \gamma(t_3)). \end{align*}   
In the regime where $(t_4-t_3)(t_2-t_1) > 0$, these two quantities are of opposite signs, completing the proof. 
\end{proof}
\bibliographystyle{plain}

\vskip0.2in
\noindent \author{\textsc{Robert Fraser}}\\
University of British Columbia, Vancouver, Canada. \\
Electronic address: \texttt{rgf@math.ubc.ca}
\vskip0.2in 
\noindent \author{\textsc{Malabika Pramanik}}\\
University of British Columbia, Vancouver, Canada. \\
Electronic address: \texttt{malabika@math.ubc.ca}


\begin{thebibliography}{99}

\bibitem{B46} Behrend, F. A.
\newblock{\em On sets of integers which contain no three terms in
              arithmetical progression}. 
\newblock Proc. Nat. Acad. Sci. U. S. A., {\bf 32}, (1946), 331-332.

\bibitem{BIT16} Bennett, M. and Iosevich, A. and Taylor K.
\newblock{\em Finite chains inside thin subsets of ℝd.Finite chains inside thin subsets of $\mathbb R^d$}. 
\newblock Anal. PDE, {\bf 9}, No. 3, (2016), 597-614. 


\bibitem{CLP} Chan, V. and {\L}aba, I. and Pramanik, M. 
\newblock{\em Point configurations in sparse sets}. 
\newblock J. d'Analyse Math., {\bf 128}, No. 1, (2016), 289-335. 

\bibitem{F97} Falconer, K. 
\newblock{\em Techniques in fractal geometry}.
John Wiley \& Sons, Ltd., Chichester, 1997, ISBN 0-471-95724-0.  

\bibitem{GI12} Greenleaf, A. and Iosevich, A. 
\newblock{\em On triangles determined by subsets of the Euclidean plane, the associated bilinear operators and applications to discrete geometry}.  
\newblock Anal. PDE, {\bf 5}, No. 2, (2012), 397-409. 



\bibitem{GILP15} Greenleaf, A. and Iosevich, A. and Liu, B. and Palsson E. 
\newblock{\em A group-theoretic viewpoint on Erdős-Falconer problems and the Mattila integral}. 
\newblock  Rev. Mat. Iberoam., {\bf 31}, No. 3,  (2015), 799–810.


\bibitem{GIP16} Greenleaf, A. and Iosevich, A. and Pramanik, M. 
\newblock{\em On necklaces inside thin subsets of $\mathbb R^d$}.
\newblock To appear in Math. Res. Lett., available at http://arxiv.org/pdf/1409.2588v1.pdf.  


\bibitem{HKKMMMS13} Harangi, V. and Keleti, T. and Kiss, G. and
              Maga, P. and M{\'a}th{\'e}, A. and Mattila,
              P. and Strenner, B. 
\newblock{\em How large dimension guarantees a given angle?}
\newblock Monatsh. Math., {\bf 171}, No. 2, (2013), 169-187.   


\bibitem{HLP} Henriot, K. and {\L}aba, I. and Pramanik, M.  
\newblock{\em On polynomial configurations in fractal sets}. 
\newblock to appear in Anal. PDE.  



\bibitem{K98} Keleti, T.
\newblock{\em A 1-dimensional subset of the reals that intersects each of
              its translates in at most a single point}. 
\newblock Real Anal. Exchange, {\bf 24}, No. 2, (1998/99), 843-844.



\bibitem{K08} Keleti, T. 
\newblock{\em Construction of one-dimensional subsets of the reals not
              containing similar copies of given patterns}.
\newblock Anal. PDE, {\bf 1}, No. 1, (2008), 29-33.  


\bibitem{LP09} {\L}aba, I. and Pramanik, M. 
\newblock{\em Arithmetic progressions in sets of fractional dimension}. 
\newblock Geom. Funct. Anal., {\bf 19}, No. 2, (2009), 429-456.  


\bibitem{M10} Maga, P. 
\newblock{\em Full dimensional sets without given patterns}. 
\newblock Real Anal. Exchange, {\bf 36}, No. 1, (2010/11), 79-90.


\bibitem{M12} M{\'a}th{\'e}, A.
\newblock{\em Sets of large dimension not containing polynomial configurations}.  
\newblock (2012), available at http://arxiv.org/pdf/1201.0548v1.pdf. 

\bibitem{SS42} Salem, R. and Spencer, D. C.
\newblock{\em On sets of integers which contain no three terms in
              arithmetical progression}.
\newblock Proc. Nat. Acad. Sci. U. S. A., {\bf 28}, (1942), 561-563. 

\bibitem{S75} Szemer{\'e}di, E.
\newblock{\em On sets of integers containing no {$k$} elements in arithmetic
              progression}.
\newblock Acta Arith., Collection of articles in memory of Juri{\u\i}
              Vladimirovi{\v{c}} Linnik, {\bf 27}, (1975), 199-245.  
\end{thebibliography}
\end{document}